\documentclass[11pt]{article}
\usepackage[dvips]{graphics}
\usepackage{epsfig}
\usepackage{multirow}
\usepackage{array}
\usepackage{extarrows}
\usepackage{graphicx}

\usepackage{comment}
\usepackage[titletoc,title]{appendix}
\numberwithin{equation}{section}
\usepackage{amsthm}
\usepackage{enumerate}
\usepackage{amsmath}

\usepackage{amssymb,amsmath}
\usepackage{latexsym}
\usepackage[usenames]{color}
\usepackage{pgf}
\usepackage[sort]{natbib}
\usepackage{hyperref}

\usepackage{subfig}
\numberwithin{equation}{section}

\newcommand{\be}{\begin{equation}}
\newcommand{\ee}{\end{equation}}
\newcommand{\beaa}{\begin{eqnarray*}}
\newcommand{\eeaa}{\end{eqnarray*}}
\newcommand{\bea}{\begin{eqnarray}}
\newcommand{\eea}{\end{eqnarray}}
\newcommand{\lbl}{\label}

\newcommand{\bei}{\begin{itemize}}
\newcommand{\eei}{\end{itemize}}

\newcommand{\ml}{\mathcal}

\newtheorem{theorem}{ \noindent T{\footnotesize HEOREM}}

\newtheorem{lemma}{ \noindent L{\footnotesize EMMA}}[section]
\newtheorem{coro}{ \noindent C{\footnotesize OROLLARY}}
\newtheorem{question}{ \noindent Q{\footnotesize UESTION}}

\newcommand{\red}{\color{red}}

\newcommand{\E}{\mathbb{E}}
\begin{document}

\title{Asymptotic Properties of Random Restricted Partitions}
\author{Tiefeng Jiang$^{1}$
and  Ke Wang$^2$\\
University of Minnesota and HKUST }

\date{}
\maketitle

\footnotetext[1]{School of Statistics, University of Minnesota, 224 Church
Street, S. E., MN 55455, USA, jiang040@umn.edu. 
The research of Tiefeng Jiang is
supported in part by NSF Grant DMS-1916014 and DMS-1406279.}

\footnotetext[2]{Department of Mathematics,
Hong Kong University of Science and Technology, Hong Kong, China,  kewang@ust.hk. The research of Ke Wang is partially supported by Hong Kong RGC grant GRF 16301618, GRF 16308219, and ECS 26304920.
}

\begin{abstract}
\noindent We study two types of probability measures on the set of integer partitions of $n$  with at most $m$ parts. The first one chooses the random partition with a chance related to its largest part only. We then obtain the limiting distributions of all of the parts together and that of the largest part as $n$ tends to infinity while $m$ is fixed or tends to infinity. In particular, if $m$ goes to infinity not too fast, the largest part satisfies the central limit theorem. The second measure is very general. It includes the Dirichlet distribution and the uniform distribution as special cases. We derive the asymptotic distributions of the parts jointly by taking limits of $n$ and $m$ in the same manner as that in the first probability measure.
\end{abstract}



\newpage

\section{Introduction}\lbl{Introduction}

The partition $\kappa$ of a positive integer $n$ is a sequence of positive integers $k_1 \ge k_2 \ge \cdots \ge k_m$ with $m\ge 1$ whose sum is $n$. The number $m$ is called the length of $\kappa$ and $k_i$ the $i$th largest part of $\kappa$.  Let $\ml{P}_n$ denote the set of partitions of $n$ and $\ml{P}_n(m)$ the set of partitions of $n$ with length \emph{at most} $m$. Thus $1\le m \le n$ and $\ml{P}_n(n)$=$\ml{P}_n$.

The set of all partitions $\ml{P}=\cup_{n\ge 1} \ml{P}_n$ is called the macrocanonical ensemble. The partitions of $n$, $\ml{P}_n=\cup_{m=1}^n \ml{P}_{n}(m)$, is called the canonical ensemble and $\ml{P}_{n}(m)$ is the microcanonical ensemble. Integer partitions have a close relationship with statistical physics (\cite{BK37, VU37, AK46}). To be more precise, a partition $\kappa \in\ml{P}_n$ can be interpreted as an assembly of particles with total energy $n$. The number of particles is the length of $\kappa$; the number of particles with energy $l$ is equal to $\# \{ j: k_j=l\}.$ Thus $\ml{P}_n(m)$ is the set of configurations $\kappa$ with a given number of particles $m$. It is known that $\ml{P}_n(m)$ corresponds to the Bose-Einstein assembly (see section 3 in \cite{AK46} for a brief discussion). Therefore the asymptotic distribution of a probability measure on $\ml{P}_n(m)$ as $n$ tends to infinity is connected to how the total energy of the system is distributed among a given number of particles.

The most natural probability measure on the integer partitions is the uniform measure. The uniform measure on $\ml{P}_n(m)$ for $m=n$ has been well-studied (see \cite{EL, Fristedt, Pittel}). However, for the other values of $m$, to our best knowledge, the whole picture is not clear yet. In the authors' previous paper \cite{JW-LB}, as a by-product of studying the eigenvalues of Laplacian-Beltrami operator defined on symmetric polynomials, the limiting distribution of $(k_1,\ldots,k_m)$ chosen uniformly from $\ml{P}_n(m)$ is derived for fixed integer $m$. This is one of the motivations resulting in this paper. As a special case of a more general measure on $\ml{P}_n(m)$ (detailed definition given in Section \ref{sec:r2} below), we obtain the asymptotic joint distribution of $(k_1,\ldots,k_m)\in \ml{P}_n(m)$ imposed with a uniform measure for $m\to \infty$ and $m=o(n^{1/3})$. It would be an intriguing question to understand the uniform measure on $\ml{P}_n(m)$ for all value of $m$. The limiting shape of the young diagram corresponding to $\ml{P}_n(m)$ with respect to uniform measure was studied in \cite{Ver96, VK85, VY03} and \cite{Petrov} for $m=n$ and for $m=c \sqrt{n}$ where $c$ is a positive constant.

Another important class of probability measure on the integer partitions is the Plancherel measure or the more general $\alpha$-Jack measure. Plancherel measure is a special case of $\alpha$-Jack measure with $\alpha=1$. It is known the both the Plancherel measure (see \cite{BDJ, BOO, Jo,O2}, a survey by \cite{O1} and the references therein) and $\alpha$-Jack measure (see for instance \cite{BO, Fulman, sho}) have a deep connection with random matrix theory.

For a fixed constant $q\in (0,1)$, the $q$-analog of the Plancherel measure, which is called the $q$-Plancherel measure, on integer partitions has been studied in \cite{Ker92, Str08,FM12}. As explain in Section 2.2 from \cite{Str08}, it is related to a probability measure $M_q^{(n)}$ on $\mathcal P_n$, where $M_q^{(n)}$ is proportional to $q^{b(\kappa)}$ for a parameter  $b(\kappa)$ of $\kappa \in \mathcal P_n$ and can be understood as the $q$-deformation of the Plancherel measure.  Indeed, it is quite natural and common to consider the $q$-versions of existing probability measures; for example, the Macdonald measure on $\mathcal P$ can be thought of as the $q$-version of the circular $\beta$-ensemble (see \cite{FR05, Macbook}). This point of view motivates us to consider a probability measure on $\mathcal P_n(m)$ that chooses $\kappa \in \mathcal P_n(m)$ proportionally to $q^{\sigma(\kappa)}$, where $\sigma(\kappa)$ is a function of $\kappa=(k_1,\ldots,k_m)$. In this paper, we set $\sigma(\kappa)=k_1$, the largest part of $\kappa$, and study the asymptotic behavior of the parts of $\kappa$ as $n$ tends to infinity. This probability measure on the microcanonical ensemble $\mathcal P_n(m)$ can also be viewed as an analog of a probability measure $\mu(\cdot)$ defined on the macrocanonical ensemble $\mathcal P$, introduced in \cite{Ver96}, where $\mu(\lambda) = c q^{|\lambda|}$ for any $\lambda \in \mathcal P$ and $|\lambda|$ is the sum of its parts.

In this paper, we consider two new probability measures on $\ml{P}_n(m)$ assuming either $m$ is fixed or $m$ tends to infinity with $n$. We investigate the asymptotic joint distributions of $(k_1,\ldots,k_m)$ as $n$ tends to infinity. We first introduce the probability measures on $\ml{P}_n(m)$ and present the main results in Section \ref{sec:r1} and \ref{sec:r2}. The proofs are given in the remaining of the paper.

\subsection{Restricted Geometric Distribution}\label{sec:r1}
The first type of random partitions on $\ml{P}_n(m)$ is defined as follows: for $\kappa=(k_1,\ldots,k_m) \in \ml{P}_n(m)$, consider the probability measure
\bea\lbl{eq:another}
P(\kappa) = c\cdot q^{k_1}
\eea
where $0<q<1$ and $c=c_{n,m}$ is the normalizing constant that $\sum_{\kappa \in \ml{P}_n(m)} P(\kappa) =1$. We call this probability measure the \emph{restricted geometric distribution}. This probability measure favors the partitions $\kappa$ with the smallest possible largest part $k_1$. Thus we concern the fluctuation of $k_1$ around $\lceil \frac{n}{m} \rceil$. The motivation to work on the measure  in \eqref{eq:another} has been stated in the Introduction.

When $m$ is a fixed integer, the main result is the following.

\begin{theorem}\lbl{another_weight} For given $m\geq 2$, let $\kappa=(k_1,\ldots,k_m) \in \ml{P}_n(m)$ be chosen with  probability $P(\kappa)$ as in (\ref{eq:another}).  For a subsequence $n\equiv j_0$ (mod $m$), define $j=j_0$ if $1\le j_0 \le m-1$ and $j=m$ if $j_0=0$. Then as $n\to\infty$, we have $\big(k_1 -\lceil \frac{n}{m} \rceil, \ldots, k_m-\lceil \frac{n}{m} \rceil )$ converges to a discrete random vector with pmf
\beaa
f(l_1, \cdots, l_m)=\frac{q^{l_1}}{\sum_{l=0}^{\infty} q^{l }\cdot |\ml{P}_{ m(l+1)-j}(m-1)|}
\eeaa
for all integers $(l_1, \cdots, l_m)$ with $l_1 \ge 0$, $l_1\geq \cdots \geq l_m$ and $\sum_{i=1}^ml_i=j-m.$
\end{theorem}

From Theorem \ref{another_weight}, we immediately obtain the limiting distribution of the largest part $k_1$, which fluctuates around  its smallest possible value $\lceil \frac{n}{m}\rceil$. As a consequence, the conditional distribution of $(k_2,\ldots,k_m)$ given the largest part $k_1$ is asymptotically a uniform distribution.
\begin{coro}\lbl{cor:first}  Given $m\geq 2$, let $\kappa=(k_1,\ldots,k_m) \in \ml{P}_n(m)$ be chosen with  probability $P(\kappa)$ as in (\ref{eq:another}).   For a subsequence $n\equiv j_0$ (mod $m$), define $j=j_0$ if $1\le j_0 \le m-1$ and $j=m$ if $j_0=0$. Then as $n \to\infty$, we have $k_1 -\lceil \frac{n}{m}\rceil$ converges to a discrete random variable with pmf
\beaa
f(l)=\frac{q^{l}\cdot |\ml{P}_{ ml+m-j}(m-1)|}{\sum_{l=0}^{\infty} q^{l }\cdot |\ml{P}_{ ml+m-j}(m-1)|},~~ l\geq 0.
\eeaa
Furthermore, the conditional distribution
of $(k_2 -\lceil \frac{n}{m}\rceil, \ldots, k_m-\lceil \frac{n}{m}\rceil)$ given $k_1=\lceil \frac{n}{m}\rceil+l_1$ $(l_1 \ge 0)$ is asymptotically a uniform distribution on the set
$\big\{(l_2,\ldots,l_m)\in \mathbb{Z}^{m-1};\, l_1 \ge l_2 \ge \ldots \ge l_m \ \mbox{and}\  l_1 + \sum_{i=2}^m l_i = j-m  \big\}.$
\end{coro}
We present the proofs of Theorem \ref{another_weight} and Corollary \ref{cor:first} in Section \ref{sec:anotherfix}.

When $m$ tends to infinity with $n$ and $m=o(n^{1/3})$, we consider the limiting distribution of the largest part $k_1$. The main result is that with proper normalization, the largest part $k_1$ converges to a normal distribution.

\begin{theorem}\lbl{surprise_result} 
Given $q\in (0,1)$, let $\kappa=(k_1,\ldots,k_m) \in \ml{P}_n(m)$ be chosen with  probability $P(\kappa)$ as in (\ref{eq:another}). Set $\lambda =-\log q>0.$ If $m=m_n\to\infty$ with  $m=o(n^{1/3})$, then $\frac{1}{\sqrt{m}}(k_1  -\lceil \frac{n}{m}\rceil - \gamma m)$
converges weakly to $N(0, \sigma^2)$ as $n\to\infty$, where
\beaa
\gamma=\frac{1}{\lambda^{2}}\int_0^{\lambda}\frac{t}{e^t-1}\,dt\ \ \mbox{and}\ \ \sigma^2=\frac{2}{\lambda^{3}}\int_0^{\lambda}\frac{t}{e^t-1}\,dt- \frac{1}{\lambda(e^{\lambda}-1)}>0.
\eeaa
\end{theorem}

The proof of Theorem \ref{surprise_result} is analytic and quite involved. We use the Laplace method to estimate the normalization constant $c=c_{n,m}$ in \eqref{eq:another}. The same analysis is applied to obtain the asymptotic distribution of the largest part $k_1$. Thanks to the Szekeres formula (see \eqref{marriage}) for the number of restricted partitions,  we first approximate $c_{n,m}$ with an integral
$$c_{n,m} \approx C(m)\cdot\int \exp(m\psi(t))\, dt$$
for some function $\psi(t)$ that has a global maximum at $t_0>0$. Thus
$$\psi(t) \approx \psi(t_0)  -\frac{1}{2} |\psi''(t_0)| t^2$$ and
\bea\lbl{eq:nor}
c_{n,m}
&\approx& C(m) e^{m\psi(t_0)}\cdot\int \exp\left(-\frac{1}{2}m |\psi''(t_0)|t^2 \right)\, dt.
\eea
The most significant contribution in the integral comes from the $t$ close to $t_0$. Indeed, the integral in \eqref{eq:nor} is reduced to a Gaussian integral as $n\to \infty$. We prove Theorem \ref{surprise_result} in Section \ref{sec:anotherinf}.

It remains to consider the conditional distribution of $(k_2,\ldots,k_m)$ given the largest part $k_1$. It is convenient to work with $k_i=\lceil \frac{n}{m}\rceil + l_i$ for $1\le i \le m$. In view of Theorem \ref{surprise_result}, let $k_1 = \lceil \frac{n}{m}\rceil + l_1$ with $l_1=\gamma m + C\cdot\sqrt{m}$. Given $l_1$,  $(l_2,\ldots,l_m)$ has a uniform distribution on the set $\{(l_2,\ldots,l_m)\in \mathbb{Z}^{m-1};\, l_1 \ge l_2 \ge \ldots \ge l_m \ \mbox{and}\  l_1 + \sum_{i=2}^m l_i = j-m \}$. We consider a linear transform
$(j_2,\ldots,j_m)=(l_1-l_2,\ldots,l_1-l_m)$. Since uniform distribution is preserved under linear transformations, $(j_2,\ldots,j_m)$ has the uniform distribution on the set
$\{(j_2,\ldots,j_m)\in \mathbb{N}^{m-1};\, j_m \ge \ldots \ge j_3\ge j_2 \ \mbox{and}\  \sum_{i=2}^m j_i = ml_1+m-j \}$. In general, the problem is related to understanding the uniform distribution on the set
\beaa
\Big\{(\lambda_2,\ldots,\lambda_m) \in \mathbb{N}^{m-1};\, \lambda_2 \ge \ldots \ge \lambda_m \ge 0 \ \mbox{and}\ \sum_{i=2}^m \lambda_i = m l_1 \Big\}.
\eeaa
To our best knowledge, it is not even clear what the limiting joint distribution of a partition chosen uniformly from $\ml{P}_{m^2}(\gamma m)$ is as $m$ tends to infinity. We raise the following question for future projects.

\begin{question}\lbl{qu1}
Given $q\in (0,1)$, let $\kappa=(k_1,\ldots,k_m) \in \ml{P}_n(m)$ be chosen with  probability $P(\kappa)$ as in (\ref{eq:another}). Assume $m$ tends to infinity with $n$ and $m=o(n^{1/3})$. Determine the asymptotic joint distribution of $(k_2,\ldots,k_m)$ given $k_1$. Furthermore, what is the limiting distribution of $(k_1,k_2, \ldots,k_m)$ as $n$ tends to infinity?
\end{question}

We have considered the limiting distribution of $\kappa \in \ml{P}_n(m)$ chosen as in (\ref{eq:another}) for $m$ fixed as well as $m=o(n^{1/3})$. It is also interesting to investigate this probability measure for other ranges of $m$.
\begin{question}\lbl{qu1}
Given $q\in (0,1)$, let $\kappa=(k_1,\ldots,k_m) \in \ml{P}_n(m)$ be chosen with  probability $P(\kappa)$ as in (\ref{eq:another}). Identify the asymptotic distribution of $\kappa$ for the entire range $1\le m \le n$.
\end{question}

\subsection{A Generalized Distribution}\label{sec:r2}

Next we consider a probability measure on $\ml{P}_n(m)$ by choosing a partition $\kappa=(k_1, \ldots, k_m) \vdash n$ with chance
\bea\lbl{eq:general}
P_n (\kappa) = c\cdot f\Big(\frac{k_1}{n},\ldots,\frac{k_m}{n}\Big)
\eea
where $c=c_{n,m}=\big({\sum_{(k_1,\ldots,k_m)\in\ml{P}_n(m)}} f(\frac{k_1}{n},\ldots,\frac{k_m}{n})\big)^{-1}$ is the normalizing constant and  $f(x_1,\ldots,x_m)$ is defined on $\overline{\nabla}_{m-1}$, the closure of $\nabla_{m-1}$. Here $\nabla_{m-1}$ is the ordered $(m-1)$-dimensional simplex defined as
\beaa
\nabla_{m-1}: = \Big\{ (y_1,\ldots,y_m) \in [0,1]^m;  y_1> y_2 > \ldots > y_{m-1} > y_m \text{ and }  y_m=1-\sum_{i=1}^{m-1}y_i\Big\}.
\eeaa
We assume $f$ is a probability density function on $\nabla_{m-1}$ and is either bounded continuous or Lipschitz on $\overline{\nabla}_{m-1}$.

When $m$ is a fixed integer, we study the limiting joint distribution of the parts of $\kappa$ chosen as in \eqref{eq:general}. The main result is the following.
\begin{theorem}\lbl{thm:general}
Let $m \ge 2$ be a fixed integer. Assume $\kappa=(k_1,\ldots, k_m) \in \ml{P}_n(m)$ is chosen as in \eqref{eq:general}, where $f$ is a probability density function on ${\nabla}_{m-1}$ and $f$ is bounded continuous on $\overline{\nabla}_{m-1}$. Then $(\frac{k_1}{n},\ldots, \frac{k_m}{n})$ converges weakly to a probability measure $\mu$ with density function $f(y_1,\ldots,y_m)$ defined on $\nabla_{m-1}$.
\end{theorem}

From Theorem \ref{thm:general}, we can immediately obtain the limiting convergence to several familiar distributions.
We say $(X_1, \ldots, X_m)$ has the \emph{symmetric Dirichlet distribution} with parameter $\alpha>0$, denoted by $(X_1,\ldots,X_m) \sim \text{Dir}(\alpha)$, if the distribution has pdf
$$\frac{\Gamma(m\alpha)}{\Gamma(\alpha)^m} x_1^{\alpha-1} \cdots x_m^{\alpha-1}$$ on the $(m-1)$-dimensional simplex
$$W_{m-1}:=\Big\{ (x_1,\ldots,x_{m-1},x_m) \in [0,1]^m;\, \sum_{i=1}^m x_i =1\Big\}$$ and zero elsewhere.

\begin{coro}\label{cor:dir}
Let $m \ge 2$ be a fixed integer. Assume $\kappa=(k_1,\ldots, k_m) \in \ml{P}_n(m)$ is chosen as in \eqref{eq:general} with  $f(x_1,\ldots,x_m)= c\cdot x_1^{\alpha-1} \cdots x_m^{\alpha-1}$ for  some $\alpha\ge 1$ and $1/c=\int_{\nabla_{m-1}} x_1^{\alpha-1} \cdots x_m^{\alpha-1} \,dx_1\ldots dx_{m-1}$, then
$$\Big(\frac{k_1}{n},\ldots, \frac{k_m}{n}\Big) \to (X_{(1)}, \ldots, X_{(m)})$$
where $(X_{(1)}, \ldots, X_{(m)})$ is the decreasing order statistics of $(X_1,\ldots,X_m) \sim \text{Dir}(\alpha).$
\end{coro}

\begin{coro}\label{cor:sphere}
Let $m \ge 2$ be a fixed integer. Assume $\kappa=(k_1,\ldots, k_m) \in \ml{P}_n(m)$ is chosen as in \eqref{eq:general} with  $f(x_1,\ldots,x_m)= c\cdot x_1^{\alpha-1} \cdots x_m^{\alpha-1}$ for some $\alpha \ge 1$ and $1/c=\int_{\nabla_{m-1}} x_1^{\alpha-1} \cdots x_m^{\alpha-1} \,dx_1\ldots dx_{m-1}$, then
$$\left(\Big(\frac{k_1}{n}\Big)^{\alpha},\ldots, \Big(\frac{k_m}{n}\Big)^{\alpha}\right) \to (Y_1, \ldots, Y_m)$$
as $n\to \infty$, where $(Y_1, \ldots, Y_m)$ has the uniform distribution on $$\Big\{(y_1,\ldots,y_m)\in[0,1]^m; \sum_{i=1}^m y_i^{1/{\alpha}} = 1, y_1 \ge\ldots \ge y_m \Big\},$$ or equivalently, $(Y_1, \ldots, Y_m)$ is the decreasing order statistics of the uniform distribution on $\{(y_1,\ldots,y_m)\in[0,1]^m; \sum_{i=1}^m y_i^{1/{\alpha}} = 1\}$.
\end{coro}

For the special case $\alpha =1$, that is, $\kappa$ is chosen uniformly from $\ml{P}_n(m)$, the conclusion of Corollary \ref{cor:sphere} is first proved in \cite{JW-LB}. The proofs of Theorem \ref{thm:general}, Corollary \ref{cor:dir} and Corollary \ref{cor:sphere} are included in Section \ref{sec:generalfix}.

When $m$ grows with $n$, we establish the limiting distribution of random restricted partitions in $\ml{P}_n(m)$. Define
\beaa
\nabla = \Big\{ (y_1, y_2, \cdots) \in [0,1]^\infty;\  y_1\geq  y_2 \geq \cdots\ \mbox{and} \sum_{i=1}^{\infty}y_i \le 1 \Big\}.
\eeaa
 Note that $\nabla_{m-1}$ can be viewed as subsets of
\beaa
\nabla_{\infty}= \Big\{(y_1,y_2,\ldots) \in [0,1]^{\infty};\  y_1\geq  y_2 \geq \cdots\ \mbox{and} \sum_{i=1}^{\infty}y_i = 1 \Big\}
\eeaa
by natural embedding. And $\nabla$ is the closure of $\nabla_{\infty}$ in $[0,1]^{\infty}$ with topology inherited from $[0,1]^{\infty}$. By Tychonoff's theorem, $\nabla_{m-1}$ and $\nabla$ are compact. Furthermore, both $\nabla_{m-1}$ and $\nabla$ are compact Polish space and thus any probability measure on $\nabla_{m-1}$ is tight. Therefore, for probability measures $\{\mu_n\}_{n\ge 1}$ and $\mu$ on $\nabla$, $\mu_n$ converges to $\mu$ weakly if all the finite-dimensional distribution of $\mu_n$ converges to the corresponding finite-dimensional distribution of $\mu$.

\begin{theorem}\label{thm:general-infinity}
Let $m=o(n^{1/3}) \to \infty$ as $n\to\infty.$ Assume $\kappa=(k_1,\ldots, k_m) \in \ml{P}_n(m)$ is chosen with probability as in \eqref{eq:general} where $f$ is a probability density function on $\nabla_{m-1}$ and is Lipschitz on $\overline{\nabla}_{m-1}$. Furthermore, assume the Lipschitz constant $\|f\|_{Lip} \le K$ for an absolute constant $K>0$. Let $(X_{m,1}, \cdots, X_{m,m})$ have  density function 
$f(y_1,\cdots,y_m)$ defined on $\nabla_{m-1}$. If $(X_{m,1}, \cdots, X_{m,m})$ converges weakly to $X$ defined on $\nabla$ as $n\to \infty$, then $(\frac{k_1}{n},\cdots, \frac{k_m}{n})$ converges weakly to $X$ as $n\to \infty$.
\end{theorem}

We will prove Theorem \ref{thm:general-infinity} in Section \ref{sec:generalinf}. We have investigated the limiting distribution of $\kappa \in \ml{P}_n(m)$ chosen as in \eqref{eq:general} for both $m$ fixed and $m=o(n^{1/3})$. It would be interesting to understand the limiting distribution of $\kappa$ for other ranges of $m$. We leave this as an open question for future research.

\begin{question}
Let $\kappa=(k_1,\ldots,k_m) \in \ml{P}_n(m)$ be chosen with  probability $P_n(\kappa)$ as in (\ref{eq:general}). Identify the asymptotic distribution of $\kappa$ for the entire range $1\le m \le n$.
\end{question}

{\bf Notation:} For $x\in \mathbb{R}$, the notation $\lceil x \rceil$ stands for the smallest integer greater than or equal to $x$. The symbol $[x]$ denotes the largest integer less than or equal to $x$. We use $\mathbb{Z}$ to be the set of all real integers. For a set $A$, the notation $\#A$ or $|A|$ stands for the cardinality of $A$. We use $c\cdot A=\{c\cdot a: a\in A\}$. Denote $\ml{P}_0(k)=1$ for convenience. For $f(n),g(n)>0$, $f(n)\sim g(n)$ if $\lim_{n\to \infty} f(n)/g(n)=1$.

\section{Proofs of Theorems \ref{another_weight} and \ref{surprise_result} and Corollary \ref{cor:first}}\lbl{sec:another}
The strategies of deriving  Theorems \ref{another_weight} and \ref{surprise_result} are different. Also the proof of Theorem  \ref{surprise_result} is relatively lengthy. For clarity, their proofs are given in two sections. In Section \ref{sec:anotherfix} we will present the proofs of  Theorems \ref{another_weight}  and Corollary \ref{cor:first}. Theorem \ref{surprise_result} will be established in Section \ref{sec:anotherinf}.

\subsection{The Proofs of Theorems \ref{another_weight}  and Corollary \ref{cor:first}}\lbl{sec:anotherfix}
In this section, $m$ is assumed to be a fixed integer. We start with a lemma concerning the number of restricted partitions $\ml{P}_n(m)$ with the largest part fixed.
\begin{lemma}\lbl{lantern_wind} Let $l \geq 0$, $m\geq 2$ and $n\geq 1$ be integers. Set $j=m+n-m\lceil \frac{n}{m} \rceil$. Then $1\leq j \leq m$.
If $0\le l \le \frac{1}{m-1}(\frac{n}{m}-m)$, we have
\bea\lbl{eq:newbaby}
\# \Big\{(k_1, k_2,\ldots,k_m) \in \ml{P}_n(m);\, k_1=  \lceil \frac{n}{m} \rceil + l \Big\}= |\ml{P}_{ m(l+1)-j}(m-1)|.
\eea
If $ \frac{1}{m-1}(\frac{n}{m}-m) < l \le n-\lceil \frac{n}{m} \rceil$, we have
\bea\lbl{eq:othervalues}
\# \Big\{(k_1, k_2,\ldots,k_m) \in \ml{P}_n(m);\, k_1=  \lceil \frac{n}{m} \rceil + l \Big\} \le |\ml{P}_{ m(l+1)-j}(m-1)|.
\eea
\end{lemma}
\begin{proof} For $\kappa=(k_1,\ldots,k_m) \in \ml{P}_n(m)$, let us rewrite $k_i = \lceil \frac{n}{m} \rceil + l_i$ for $1\le i \le m$.   By assumption, $l_1=l \ge 0$.    Since $\kappa \vdash n$, we have $l_1 \ge l_2 \ge \ldots \ge l_m \geq -\lceil \frac{n}{m} \rceil$ and $l_1 + \sum_{i=2}^m l_i = n-m \lceil \frac{n}{m} \rceil  = j-m$ by assumption. Therefore,
\beaa
& & \#\Big\{(k_1, k_2,\ldots,k_m) \in \ml{P}_n(m);\, k_1=  \lceil \frac{n}{m} \rceil + l_1 \Big\}\\
& = &\#\Big\{(l_2,\ldots,l_m)\in \mathbb{Z}^{m-1};\, l_1 \ge l_2 \ge \ldots \ge l_m \geq -\lceil \frac{n}{m}\rceil\ \mbox{and}\  l_1 + \sum_{i=2}^m l_i = j-m  \Big\}\\
& = & \#\Big\{(j_2,\ldots,j_m)\in \mathbb{Z}^{m-1};\, l_1 +\lceil \frac{n}{m}\rceil \ge j_m \ge \ldots \ge j_2\geq 0\ \mbox{and}\  \sum_{i=2}^m j_i = m(l_1+1)-j \Big\}
\eeaa
by the transform $j_i = l_1 - l_i$ for $2\le i \le m$.

Assume $0\le l \le \frac{1}{m-1}(\frac{n}{m}-m)$. If $j_m \ge \ldots \ge j_2\geq 0$ and  $\sum_{i=2}^m j_i = m(l_1+1)-j$, then
\beaa
j_m\leq \sum_{i=2}^m j_i = m(l_1+1)-j\leq m(l_1+1)\leq l_1+\lceil \frac{n}{m} \rceil
\eeaa
by assumption, the notation $l_1=l$ and the fact $\lceil x\rceil\geq x$ for any $x\in \mathbb{R}$. It follows that the left hand side of (\ref{eq:newbaby}) is identical to
\beaa
& & \#\Big\{(j_2,\ldots,j_m)\in \mathbb{Z}^{m-1};\, j_m \ge \ldots \ge j_2\geq 0\ \mbox{and}\  \sum_{i=2}^m j_i = m(l_1+1)-j \Big\}\\
&=& |\ml{P}_{ m(l+1)-j}(m-1)|.~~~~~~~~~~~~~~~~~~~~~~~~~~~~~~~~~~~~~~~~~~~~~~~~~~~~~~~~~~~~~~~~~~~~~~~~~~~~~~~
\eeaa

For $ \frac{1}{m-1}(\frac{n}{m}-m) +1 \le l \le n-\lceil \frac{n}{m} \rceil$, the upper bound \eqref{eq:othervalues} follows directly from the definitions of the sets.
\end{proof}

Now we are ready to present the proof of Theorem \ref{another_weight}.

\begin{proof}[Proof of Theorem \ref{another_weight}] First, it is easy to check that for the subsequence
$n\equiv j_0$ (mod $m$), if we define $j=j_0$ if $1\le j_0 \le m-1$ and $j=m$ if $j_0=0$, then $j=m+n-m\lceil \frac{n}{m} \rceil$.
Set
\bea\lbl{glad_carpet}
M_n=\Big[\frac{1}{m-1}\Big(\frac{n}{m}-m\Big)\Big].
\eea
We first estimate the normalizing constant $c$ in (\ref{eq:another}).
\beaa
1&=& \sum_{\kappa \in \ml{P}_n(m)} P(\kappa) = c\cdot \sum_{k_1 =\lceil \frac{n}{m} \rceil}^n \sum_{(k_1, k_2,\ldots,k_m) \vdash n}  q^{k_1} \\
&=& c\cdot  \sum_{l=0}^{n-\lceil \frac{n}{m} \rceil} q^{\lceil \frac{n}{m} \rceil +l }\sum_{(\lceil \frac{n}{m} \rceil+l,k_2,\ldots,k_m) \vdash n} 1.
\eeaa
We first show that, as $n$ tends to infinity,
\beaa
& &\sum_{l=0}^{n-\lceil \frac{n}{m} \rceil} q^{\lceil \frac{n}{m} \rceil +l }\sum_{(\lceil \frac{n}{m} \rceil+l,k_2,\ldots,k_m) \vdash n} 1 \\
& &\sim \sum_{l=0}^{M_n} q^{\lceil \frac{n}{m} \rceil +l }\sum_{(\lceil \frac{n}{m} \rceil+l,k_2,\ldots,k_m) \vdash n} 1.
\eeaa
By Lemma \ref{lantern_wind},
\beaa
\frac{\sum_{M_n+1}^{n-\lceil \frac{n}{m} \rceil} q^{\lceil \frac{n}{m} \rceil +l }\sum_{(\lceil \frac{n}{m} \rceil+l,k_2,\ldots,k_m) \vdash n} 1}{\sum_{l=0}^{M_n} q^{\lceil \frac{n}{m} \rceil +l }\sum_{(\lceil \frac{n}{m} \rceil+l,k_2,\ldots,k_m) \vdash n} 1}
&\le& \frac{\sum_{M_n+1}^{n-\lceil \frac{n}{m} \rceil} q^l \cdot |\ml{P}_{ m(l+1)-j}(m-1)|}{\sum_{l=0}^{M_n} q^l \cdot |\ml{P}_{ m(l+1)-j}(m-1)|}\\
&\sim& \frac{\sum_{M_n+1}^{n-\lceil \frac{n}{m} \rceil} q^l \binom{ml+m-j-1}{m-2} }{\sum_{l=0}^{M_n} q^l \binom{ml+m-j-1}{m-2}},
\eeaa
where the last equality follows from \eqref{eq:size}. Note that the series $\sum_{s=1}^{\infty} s^{m-2} q^s$ converges for $0<q<1$. We have
\beaa
\frac{\sum_{M_n+1}^{n-\lceil \frac{n}{m} \rceil} q^l \binom{ml+m-j-1}{m-2} }{\sum_{l=0}^{M_n} q^l \binom{ml+m-j-1}{m-2}} =O\Big( \frac{\sum_{M_n+1}^{n-\lceil \frac{n}{m} \rceil} q^l l^{m-2}}{\sum_{l=0}^{M_n} q^l l^{m-2}}\Big) = o(1).
\eeaa
Therefore, one obtains the normalizing constant
\bea\lbl{Great_bug}
c\sim \frac{1}{ q^{\lceil \frac{n}{m} \rceil} \sum_{l=0}^{M_n} q^{l } \cdot |\ml{P}_{ m(l+1)-j}(m-1)|}.
\eea

Now we study the limiting joint distribution of the parts $$(k_1, k_2, \ldots, k_m) = \Big(\lceil \frac{n}{m} \rceil + l_1, \lceil \frac{n}{m} \rceil + l_2, \ldots,  \lceil \frac{n}{m} \rceil + l_m\Big).$$
First, we claim that it is enough to consider $l_1$ to be any fixed integer from $\{ 0, 1, 2, \ldots\}$. Indeed, for any $L=L(n) \to \infty$ as $n \to \infty$, it follows from \eqref{eq:size} that
\beaa
P\Big(k_1 \ge  \lceil \frac{n}{m} \rceil + L\Big)&=&\sum_{l=L}^{M_n} P(k_1 = \lceil \frac{n}{m} \rceil + l)\\
 &=& \sum_{l=L}^{M_n}c \cdot q^{\lceil \frac{n}{m} \rceil + l} |\ml{P}_{ ml+m-j}(m-1)|\\
 &\sim& c\cdot q^{\lceil \frac{n}{m} \rceil} \sum_{l=L}^{M_n} \frac{\binom{ml+m-j-1}{m-2}}{(m-1)!} q^{l}.
\eeaa
Plugging in the normalizing constant $c$ in \eqref{Great_bug} and letting $L \to \infty$, we have
\beaa
P(k_1 \ge  \lceil \frac{n}{m} \rceil + L) &=& O\Big(\frac{\sum_{l=L}^{M_n} l^{m-2} q^{l}}{\sum_{l=0}^{M_n} q^{l } \cdot |\ml{P}_{ ml+m-j}(m-1)|}\Big)\\
 &=& o(1)
\eeaa
as $n \to \infty$. The last equality follows from the fact that the series $\sum_{s=1}^{\infty} s^{m-2} q^s$ converges for $0<q<1$.
Likewise, we have as $n$ tends to infinity,
\bea\lbl{eq:constant}
c \sim q^{-\lceil \frac{n}{m} \rceil} \frac{1}{\sum_{l=0}^{\infty} q^l \cdot |\ml{P}_{ ml+m-j}(m-1)|}.
\eea
Therefore, for any given $l_1 = 0, 1, 2, \ldots$, we conclude that
\beaa
&& P\Big(k_1 = \lceil \frac{n}{m} \rceil + l_1, k_2 = \lceil \frac{n}{m} \rceil + l_2, \ldots, k_m=\lceil \frac{n}{m} \rceil + l_m\Big)\\
 &= &c \cdot q^{\lceil \frac{n}{m} \rceil + l_1} \to \frac{q^{l_1}}{\sum_{l=0}^{\infty} q^{l } \cdot |\ml{P}_{ ml+m-j}(m-1)|}.
\eeaa
The proof is completed.
\end{proof}

\begin{proof}[Proof of Corollary \ref{cor:first}]
By Theorem \ref{another_weight}, it is enough to consider $k_1 = \lceil \frac{n}{m}\rceil + l$ for $l \in \{0,1,2,\ldots\}$ in the limiting distribution. From \eqref{eq:another},  Lemma \ref{lantern_wind} and \eqref{eq:constant},
\beaa
P\Big(k_1 =  \lceil \frac{n}{m}\rceil + l\Big) &=& c\cdot q^{\lceil \frac{n}{m}\rceil + l} \sum_{(\lceil \frac{n}{m} \rceil+l,k_2,\ldots,k_m) \vdash n} 1\\
&=& c\cdot q^{\lceil \frac{n}{m}\rceil + l} \cdot |\ml{P}_{ ml_1+m-j}(m-1)|\\
&\to & \frac{q^{l}\cdot |\ml{P}_{ ml+m-j}(m-1)|}{\sum_{l=0}^{\infty} q^{l }\cdot |\ml{P}_{ ml+m-j}(m-1)|}
\eeaa
as $n\to \infty$.

Furthermore, since
\beaa
&&P(k_2 -\lceil \frac{n}{m}\rceil=l_2,\ldots, k_m -\lceil \frac{n}{m}\rceil=l_m ~|~ k_1 -\lceil \frac{n}{m}\rceil=l_1) \\
& &\quad = \frac{P(k_1 -\lceil \frac{n}{m}\rceil=l_1, k_2 -\lceil \frac{n}{m}\rceil =l_2,\ldots, k_m -\lceil \frac{n}{m}\rceil=l_m)}{P(k_1 -\lceil \frac{n}{m}\rceil=l_1)}\\
& & \quad \sim \frac{f(l_1,\ldots,l_m)}{f(l_1)} = \frac{1}{|\ml{P}_{ ml_1+m-j}(m-1)|}
\eeaa
as $n\to \infty$, it follows immediately the conditional distribution
of $(k_2 -\lceil \frac{n}{m}\rceil, \ldots, k_m-\lceil \frac{n}{m}\rceil)$ given $k_1=\lceil \frac{n}{m}\rceil+l_1$ $(l_1 \ge 0)$ is asymptotically a uniform distribution on the set
$\big\{(l_2,\ldots,l_m)\in \mathbb{Z}^{m-1};\, l_1 \ge l_2 \ge \ldots \ge l_m \ \mbox{and}\  l_1 + \sum_{i=2}^m l_i = j-m  \big\}.$ This completes the proof.
\end{proof}


\subsection{The Proof of Theorem \ref{surprise_result}}\lbl{sec:anotherinf}
Szekeres formula (see \citet{Sz1, Sz2}; see also \citet{Can} and \cite{Romik}) says that
\bea\lbl{marriage}
|\ml{P}_{n}(k)| \sim \frac{f(u)}{n}e^{\sqrt{n} g(u)}
\eea
uniformly for $k\ge n^{1/6}$, where $u=k/\sqrt{n}$,
\bea
& & f(u)=\frac{v}{2^{3/2}\pi u}\big(1-e^{-v}-\frac{1}{2}u^2e^{-v}\big)^{-1/2}, \lbl{snow_gone}\\
& & g(u)=\frac{2v}{u}-u\log (1-e^{-v}),\lbl{spring_come}
\eea
and $v=v(u)$ is determined implicitly by
\bea\lbl{blue_star}
u^2=\frac{v^2}{\int_0^v\frac{t}{e^t-1}\,dt}.
\eea

We start with a technical lemma that will be used in the proof of Theorem \ref{surprise_result} later.
\begin{lemma}\lbl{London} Let $\lambda>0$ be given.  Define $\psi(t)= \frac{g(t)}{t}-\frac{\lambda}{t^2}$ for $t>0.$ Then
\beaa
t_0:=\frac{\lambda}{(\int_0^{\lambda}\frac{t}{e^t-1}\,dt)^{1/2}}\ \ \mbox{satisfies}\ \ \psi''(t_0) = -\frac{2\lambda(e^{\lambda}-1)}{t_0^4(e^{\lambda}-1-\frac{1}{2}t_0^2)} <0.
\eeaa
Further, $\psi'(t_0)=0$, $\psi(t)$ is strictly increasing on $(0, t_0]$ and strictly decreasing on $[t_0, \infty)$.
\end{lemma}
\begin{proof} Trivially, the function $\frac{t}{e^t-1}=(\sum_{i=1}^{\infty}\frac{t^{i-1}}{i!})^{-1}$ is positive and decreasing in $t\in (0, \infty).$
It follows  that $v=v(u)>0$ for all $u\in (0, \infty)$ and
$$\frac{v^2}{u^2} = \int_0^v\frac{t}{e^t-1}\,dt > \frac{v^2}{e^v-1}.$$
Thus $e^v - 1- u^2 >0$. In particular,
\bea\lbl{bad_tree}
e^v-1-\frac{1}{2}u^2  >0.
\eea
By taking derivative from \eqref{blue_star}, we get
$$2v\cdot v' = 2u \int_0^v\frac{t}{e^t-1}\,dt + u^2 \frac{v\cdot v'}{e^v-1}.$$ This implies that
$\frac{v'}{e^v-1} = \frac{2v'}{u^2} - \frac{2v}{u^3}$, or equivalently,
\bea\lbl{eye_oil}
v'= \frac{v}{u} + \frac{uv}{2(e^v-1 - \frac{1}{2}u^2)}.
\eea
 Consequently,
$v'=v'(u)>0$ for all $u>0$, and thus $v(u)$ is strictly increasing on $(0, \infty)$. Take derivative on $g(u)$ in (\ref{spring_come}),  and use (\ref{blue_star}) and (\ref{eye_oil}) to see
\bea\lbl{xenix}
& & g'(u)= -\log(1-e^{-v});\\
& & g''(u)= -\frac{v' e^{-v}}{1-e^{-v}} = -\frac{v/u}{e^v-1-\frac{1}{2}u^2}.\nonumber
\eea
Therefore
\bea\lbl{Beijing}
\Big(\frac{g(u)}{u}\Big)'= \frac{ug'(u)-g(u)}{u^2}
\eea
and
\beaa
\Big(\frac{g(u)}{u}\Big)'' &= &\frac{g''(u)}{u}-2\frac{g'(u)}{u^2} + 2\frac{ g(u)}{u^3} \\
&=& \frac{v}{u^4} \Big( 4- \frac{u^2}{e^v-1-\frac{1}{2}u^2} \Big).
\eeaa
With the above preparation, we now study $\psi(t)$ (we switch the variable ``$u$" to ``$t$").
\bea
\psi''(t)&=&\Big(\frac{g(t)}{t} -\frac{\lambda}{t^2}\Big)''  \nonumber\\
&=& \frac{v}{t^4} \Big( 4- \frac{t^2}{e^v-1-\frac{1}{2}t^2} \Big) - \frac{6\lambda}{t^4}  \nonumber\\
&=& \frac{1}{t^4} \Big( 4v-6\lambda- \frac{v \cdot t^2}{e^v-1-\frac{1}{2}t^2}\Big). \lbl{teeth_appointment}
\eea
The assertions in (\ref{xenix}) and (\ref{Beijing}) imply
\beaa
\Big(\frac{g(t)}{t}\Big)'=\frac{-t^2\log(1-e^{-v})-tg(t)}{t^3}=-\frac{2v}{t^3}.
\eeaa
Thus, $\psi'(t)=\frac{2(\lambda-v)}{t^3}$. Thus, the stable point $t_0$ of $\psi(t)$ satisfies that $v(t_0)=\lambda$. This implies that $\psi(t)$ is strictly increasing on $(0, t_0]$ and strictly decreasing on $[t_0, \infty)$. It is not difficult to see from (\ref{blue_star}) that
\beaa
t_0=\frac{\lambda}{(\int_0^{\lambda}\frac{t}{e^t-1}\,dt)^{1/2}}.
\eeaa
Plug this into (\ref{teeth_appointment}) to get
\beaa
\psi''(t_0) &=& -\frac{1}{t_0^4}\Big( 2\lambda +\frac{\lambda \cdot t_0^2}{e^{\lambda}-1-\frac{1}{2}t_0^2}\Big)\\
& = & -\frac{2\lambda(e^{\lambda}-1)}{t_0^4(e^{\lambda}-1-\frac{1}{2}t_0^2)}<0
\eeaa
by (\ref{bad_tree}).
\end{proof}

\begin{proof}[Proof of Theorem \ref{surprise_result}]  Let $M_n=[\frac{1}{m-1}(\frac{n}{m}-m)]$ as in (\ref{glad_carpet}).   The assumption $m=o(n^{1/3})$ implies
\bea\lbl{beckey}
\lim_{n\to\infty}\frac{M_n}{m}=\infty.
\eea
Similar to (\ref{Great_bug}), we first claim that the normalization constant
\beaa
c\sim \frac{1}{ q^{\lceil \frac{n}{m} \rceil} \sum_{l=0}^{M_n} q^{l } \cdot |\ml{P}_{ m(l+1)-j}(m-1)|}.
\eeaa
Indeed, from Lemma \ref{lantern_wind},
\beaa
\frac{1}{c} &=& \sum_{l=0}^{n-\lceil \frac{n}{m} \rceil} q^{\lceil \frac{n}{m} \rceil +l }\sum_{(\lceil \frac{n}{m} \rceil+l,k_2,\ldots,k_m) \vdash n} 1 \\
&=&\sum_{l=0}^{M_n} q^{\lceil \frac{n}{m} \rceil +l } \cdot |\ml{P}_{ m(l+1)-j}(m-1)| + \sum_{l=M_n+1}^{n-\lceil \frac{n}{m} \rceil} q^{\lceil \frac{n}{m} \rceil +l } \sum_{(\lceil \frac{n}{m} \rceil+l,k_2,\ldots,k_m) \vdash n} 1
\eeaa
and
\beaa
\sum_{l=M_n+1}^{n-\lceil \frac{n}{m} \rceil} q^{\lceil \frac{n}{m} \rceil +l } \sum_{(\lceil \frac{n}{m} \rceil+l,k_2,\ldots,k_m) \vdash n} 1
&\le& \sum_{l=M_n+1}^{n-\lceil \frac{n}{m} \rceil} q^{\lceil \frac{n}{m} \rceil +l } \cdot |\ml{P}_{ m(l+1)-j}(m-1)|\\
&=& \sum_{l=M_n+2}^{n-\lceil \frac{n}{m} \rceil+1} q^{\lceil \frac{n}{m} \rceil +l } \cdot |\ml{P}_{lm-j}(m-1)|.
\eeaa
Observe that $|\ml{P}_{lm-j}(m-1)| \leq |\ml{P}_{lm}(lm)| \leq e^{K\sqrt{lm}}$ for some constant $K>0$ by Hardy-Ramanujan formula. Therefore,
\beaa
\sum_{l=M_n+1}^{n-\lceil \frac{n}{m} \rceil} q^{\lceil \frac{n}{m} \rceil +l } \sum_{(\lceil \frac{n}{m} \rceil+l,k_2,\ldots,k_m) \vdash n} 1
&\le& q^{\lceil \frac{n}{m} \rceil} \sum_{l=M_n}^{\infty} e^{-\lambda l + K\sqrt{lm}}\\
&\le& q^{\lceil \frac{n}{m} \rceil} \sum_{l=M_n}^{\infty} e^{-\lambda l/2 } \le q^{\lceil \frac{n}{m} \rceil} \frac{e^{-\lambda M_n/2 }}{1-e^{-\lambda/2}}\\
&=& o\Big(\sum_{l=0}^{M_n} q^{\lceil \frac{n}{m} \rceil +l } \cdot |\ml{P}_{ m(l+1)-j}(m-1)|\Big)
\eeaa
for $n$ sufficiently large. Hence, without loss of generality, we have
\beaa
P\Big(k_1 =  \lceil \frac{n}{m}\rceil + l\Big)
= \frac{q^{l} |\ml{P}_{ m(l+1)-j}(m-1)|}{  \sum_{l=0}^{M_n} q^{l } \cdot |\ml{P}_{ m(l+1)-j}(m-1)|}
\eeaa
for $l=0,1,2,\cdots, M_n$, where $j=m+n-m\lceil \frac{n}{m}\rceil$ and $1\leq j \leq m$. Thus,
\bea\lbl{aha}
P\Big(k_1 \leq  \lceil \frac{n}{m}\rceil + m\xi\Big)
= \frac{\sum_{l=1}^{[m\xi]+1} q^{l} \cdot |\ml{P}_{ lm-j}(m-1)|}{\sum_{l=1}^{M_n+1} q^{l } \cdot |\ml{P}_{lm-j}(m-1)|}
\eea
for any $\xi\geq 0$.

In the following, we first apply a fine analysis to estimate the denominator
$$\sum_{l=1}^{M_n+1} q^{l } \cdot |\ml{P}_{lm-j}(m-1)|.$$
We divide the range of summation into five parts: $1\le l \le c m$, $Cm\le l \le M_n$,  $cm\leq l < \gamma m -\sqrt{m}\log m$, $\gamma m+\sqrt{m}\log m < l \leq Cm$ and $\gamma m -\sqrt{m}\log m \leq l \leq \gamma m +\sqrt{m}\log m$ for some proper constants $c,C>0$ and $\gamma = t_0^{-2}$ (recall $t_0$ in Lemma \ref{London}). The most significant contribution in the summation comes from the range $\gamma m -\sqrt{m}\log m \leq l \leq \gamma m +\sqrt{m}\log m$ and others are negligible. The estimation for the numerator is similar.

{\it Step 1: Two rough tails are negligible}. First, by Hardy-Ramanujan formula, there exists a constant $K>0$ such that
\beaa
|\ml{P}_{lm-j}(m-1)| \leq |\ml{P}_{lm}(lm)| \leq e^{K\sqrt{lm}}
\eeaa
for $l\geq 1$ as $n$ is large. Set $\lambda=-\log q>0$. It follows that
\bea\lbl{pig}
\sum_{l=Cm}^{M_n+1}q^{l } \cdot |\ml{P}_{lm-j}(m-1)| \leq \sum_{l=Cm}^{\infty} e^{-\lambda l+K\sqrt{ml}}\leq \sum_{l=Cm}^{\infty} e^{-\lambda l/2} \leq \frac{1}{1-e^{-\lambda/2}}
\eea
for all $l\geq (\frac{4K^2}{\lambda^2})m$, which is satisfied if $C>\frac{4K^2}{\lambda^2}.$ Similarly, for the same $K$ as above,
\bea\lbl{brain_head}
\sum_{l=1}^{cm}q^{l } \cdot |\ml{P}_{lm-j}(m-1)| & \leq & \sum_{l=1}^{cm} q^{l } \cdot |\ml{P}_{[cm^2]}(m)| \nonumber\\
 &\leq & (cm)\cdot|\ml{P}_{[cm^2]}(m)| \leq (cm) \cdot e^{\sqrt{c}Km}
\eea
for all $c>0$ as $n$ is sufficiently large.

In the rest of the proof, the variable $n$ will be  hidden in $m=m_n$ and $j=j_n$. Keep in mind that $m$ is sufficiently large when we say  ``$n$ is sufficiently large".
We set two parameters
\bea\lbl{bigc}
C=\max\Big\{\frac{8K^2}{\lambda^2}, 2\gamma \Big\};
\eea
\bea\lbl{littlec}
c=\min\Big\{\frac{\psi(t_0)^2}{16K^2}, \frac{\gamma}{2} \Big\}.
\eea

{\it Step 2: Two refined tails are negligible}. Recall $t_0$ in Lemma \ref{London}. Define $\gamma=t_0^{-2}$ and
\bea
& & \Omega_1=\{l\in \mathbb{N};\, cm\leq l < \gamma m -\sqrt{m}B\},\ \ \ \Omega_2=\{l\in \mathbb{N};\, \gamma m -\sqrt{m}B\leq l \leq \gamma m +\sqrt{m}B\},\nonumber\\
& & \Omega_3=\{l\in \mathbb{N};\, \gamma m+\sqrt{m}B < l \leq Cm\}, \lbl{soil_white}
\eea
with $B=\log m$, where $c\in (0, \gamma)$ and $C>\gamma$ by \eqref{littlec} and \eqref{bigc}. The limit in (\ref{beckey}) asserts that $\Omega_2 \subset \{1,2,\cdots, M_n\}$ as $n$ is large.
Then
\bea\lbl{hayes}
\sum_{l=cm}^{Cm}q^{l } \cdot |\ml{P}_{lm-j}(m-1)|=\sum_{i=1}^3\sum_{l\in \Omega_i}q^{l } \cdot |\ml{P}_{lm-j}(m-1)|.
\eea
Easily,
\bea\lbl{more_water}
\sum_{l\in \Omega_1\cup \Omega_3}q^{l } \cdot |\ml{P}_{lm-j}(m-1)| \leq \sum_{l\in \Omega_1\cup \Omega_3}q^{l } \cdot |\ml{P}_{lm}(m)|.
\eea
Take $n=lm$ and $k=m$ in (\ref{marriage}), we get
\bea\lbl{lake_five}
|\ml{P}_{lm}(m)| \sim \frac{f(u)}{lm}e^{\sqrt{lm} g(u)}
\eea
uniformly for all $cm\leq l \leq Cm$ where $u=\big(\frac{m}{l})^{1/2}$. Notice
\beaa
q^{l} \cdot |\ml{P}_{lm}(m)| \sim \frac{f(u)}{lm}e^{-\lambda l+\sqrt{lm} g(u)}.
\eeaa
Consider function $-\lambda x+\sqrt{x m}\cdot g\big((mx^{-1})^{1/2}\big)$
for $x\in [cm,  Cm]$. Set $t=t_x=(mx^{-1})^{1/2}$. Then
\bea
-\lambda x+\sqrt{x m}\cdot g\big((mx^{-1})^{1/2}\big)
&=& -\frac{\lambda m}{t^2}+m\frac{g(t)}{t} \nonumber\\
&= & m\Big(\frac{g(t)}{t}-\frac{1}{t^2}\lambda\Big).\lbl{car_ship}
\eea
By (\ref{snow_gone}) and (\ref{spring_come}), $f(x)$ is a continuous function on $[C^{-1/2}, c^{-1/2}]$. Therefore, $\frac{f((mj^{-1})^{1/2})}{mj} =O(m^{-2})$ uniformly for all $j\in \Omega_1\cup \Omega_3$, which together with (\ref{more_water}) yields
\bea
& & \sum_{l\in \Omega_1\cup \Omega_3}q^{l } \cdot |\ml{P}_{lm-j}(m-1)| \nonumber\\
& \leq  & O\Big(\frac{1}{m^2}\Big)\sum_{l\in \Omega_1\cup \Omega_3}\exp\Big[m\Big(\frac{g(t_l)}{t_l}-\frac{\lambda}{t_l^2}\Big)\Big] \nonumber\\
& \leq &  O\Big(\frac{1}{m}\Big)\cdot\exp\Big[m\max_{l\in \Omega_1\cup \Omega_3}\Big(\frac{g(t_l)}{t_l}-\frac{\lambda}{t_l^2}\Big)\Big]. \lbl{Road_far}
\eea
Now
\beaa
\max_{l\in \Omega_1\cup \Omega_3}\Big(\frac{g(t_l)}{t_l}-\frac{\lambda}{t_l^2}\Big)
=\max_{l\in \Omega_1\cup \Omega_3}\Big\{\psi\Big(\sqrt{\frac{m}{l}}\Big)\Big\}.
\eeaa
Evidently,
\beaa
& & \Big\{\sqrt{\frac{m}{l}},\, l \in \Omega_1\Big\} \subset \Big[\Big(\frac{m}{\gamma m- \sqrt{m}\log m}\Big)^{1/2},\, \frac{1}{\sqrt{c}}\Big] \subset (t_0, \infty);\\
& & \Big\{\sqrt{\frac{m}{l}},\, l \in \Omega_3\Big\} \subset \Big[\frac{1}{\sqrt{C}}, \Big(\frac{m}{\gamma m+ \sqrt{m}\log m}\Big)^{1/2}\Big] \subset (0, t_0).
\eeaa
Recall Lemma \ref{London}, $\psi(t)= \frac{g(t)}{t}-\frac{\lambda}{t^2}$ is increasing  $(0, t_0]$ and decreasing in $[t_0, \infty).$ It follows that
\beaa
& & \max_{l\in \Omega_1\cup \Omega_3}\Big(\frac{g(t_l)}{t_l}-\frac{\lambda}{t_l^2}\Big)\\
& \leq & \max\Big\{\psi\Big(\frac{\sqrt{m}}{\sqrt{\gamma m- \sqrt{m}\log m}}\Big),\, \psi\Big(\frac{\sqrt{m}}{\sqrt{\gamma m+ \sqrt{m}\log m}}\Big)\Big\}.
\eeaa
Notice
\beaa
 \Big(\frac{\sqrt{m}}{\sqrt{\gamma m\pm \sqrt{m}\log m}}-t_0\Big)^2
& = & \Big[\frac{1}{\sqrt{\gamma}}\Big(1\pm \frac{\log m}{\gamma\sqrt{m}}\Big)^{-1/2}-t_0\Big]^2\\
& = & \frac{(\log m)^2}{4\gamma^3m}(1+o(1)).
\eeaa
From (\ref{shaoping}), we see that
\beaa
\psi\Big(\frac{\sqrt{m}}{\sqrt{\gamma m\pm \sqrt{m}\log m}}\Big)=\psi(t_0) -L\frac{(\log m)^2}{m} +O\big(m^{-3/2}(\log m)^3\big)
\eeaa
as $n$ is large, where $L=\frac{|\psi''(t_0)|}{8\gamma^3}>0.$ This joins (\ref{Road_far}) to yield that
\beaa
\frac{1}{\sqrt{m}}e^{-m\psi(t_0)}\sum_{l\in \Omega_1\cup \Omega_3}q^{l } \cdot |\ml{P}_{lm-j}(m-1)|  \leq  e^{-(L/2)(\log m)^2}
\eeaa and thus
\bea \lbl{air_exit}
 \sum_{l\in \Omega_1\cup \Omega_3}q^{l } \cdot |\ml{P}_{lm-j}(m-1)|  \leq  \sqrt{m} e^{m\psi(t_0)-(L/2)(\log m)^2}
\eea
as $n$ is large.

{\it Step 3. The estimate of $\sum_{j\in \Omega_2}$}. Take $n=lm-j$ and $k=m-1$ in (\ref{marriage}), we get
\beaa
|\ml{P}_{ml-j}(m-1)| \sim \frac{f(u)}{ml-j}e^{\sqrt{ml-j}\, g(u)}
\eeaa
uniformly for all $cm\leq l \leq Cm$ where $u=\frac{m-1}{\sqrt{lm-j}}$. By continuity,
\bea\lbl{red_army}
\frac{f(u)}{ml-j} \sim t_0^2f(t_0)\cdot \frac{1}{m^2}
\eea
uniformly for all $l\in \Omega_2$. Consequently,
\bea
& & \sum_{l\in \Omega_2}q^{l } \cdot |\ml{P}_{lm-j}(m-1)| \nonumber\\
&= & (1+o(1))\frac{t_0^2f(t_0)}{m^2}\sum_{l\in \Omega_2}\exp\Big\{-\lambda l+\sqrt{lm-j}\cdot g\Big(\frac{m-1}{\sqrt{lm-j}}\Big)\Big\}\nonumber\\
& \sim & \frac{t_0^2f(t_0)}{m^2}e^{-\lambda j/m}\sum_{l\in \Omega_2}\exp\Big\{-\frac{\lambda(m-1)^2}{mt_l^2}+\frac{m-1}{t_l}g(t_l)\Big\}
\eea
by setting  $t_x=(m-1)/\sqrt{mx-j}$ for $x\geq 2$ (recall $1\leq j \leq m$), and hence $x=\frac{j}{m}+\frac{(m-1)^2}{mt_x^2}$. It is easy to verify that
\bea\lbl{donkey_rabbit}
\max_{l\in \Omega_2}|t_l-t_0|=O\Big(\frac{\log m}{\sqrt{m}}\Big)
\eea
as $n\to\infty$. We then have
\bea\lbl{god_amen}
& & \sum_{l\in \Omega_2}q^{l } \cdot |\ml{P}_{lm-j}(m-1)| \nonumber\\
& \sim & \frac{t_0^2f(t_0)}{m^2}e^{\lambda t_0^{-2}-(\lambda j/m)}\sum_{l\in \Omega_2}\exp\Big\{(m-1)\big(\frac{g(t_l)}{t_l}-\frac{\lambda}{t_l^2}\big)\Big\} \lbl{nice_novel}.
\eea
Recall Lemma \ref{London}. Since $\psi'(t_0)=0$ and $\psi''(t_0)<0$, it is seen  from the Taylor's expansion and (\ref{donkey_rabbit}) that
\bea\lbl{shaoping}
\psi(t_x)=\psi(t_0) + \frac{1}{2}\psi''(t_0)(t_x-t_0)^2 + O(m^{-3/2}(\log m)^3)
\eea
uniformly for all $x \in \Omega_2$. It follows that
\beaa
& & \sum_{l\in \Omega_2}\exp\Big[(m-1)\Big(\frac{g(t_l)}{t_l}-\frac{\lambda}{t_l^2}\Big)\Big] \nonumber\\
& = & (1+o(1))\cdot e^{(m-1)\psi(t_0)}\sum_{l\in \Omega_2}\exp\Big[\frac{1}{2}\psi''(t_0)(t_l-t_0)^2m\Big].
\eeaa
It is trivial to check that
\beaa
\frac{m-1}{\sqrt{mx-j}}=\frac{m-1}{\sqrt{mx}} +\frac{j}{2\gamma^{3/2}m^2} + O\Big(\frac{\log m}{m^2}\Big)
\eeaa
uniformly for all $x \in \Omega_2$. Therefore,
\beaa
m\Big(\frac{m-1}{\sqrt{mx-j}}-t_0\Big)^2
& = & m\Big(\frac{m-1}{\sqrt{mx}}-t_0\Big)^2 + \frac{j}{\gamma^{3/2}m}\Big(\frac{m-1}{\sqrt{mx}}-t_0\Big) + O\Big(\frac{\log m}{\sqrt{m}}\Big)\\
& = & m\Big(\frac{m-1}{\sqrt{mx}}-t_0\Big)^2 + o(1)
\eeaa
uniformly for all $x \in \Omega_2$ by (\ref{donkey_rabbit}). This tells us that
\bea\lbl{peach1}
& & \sum_{l\in \Omega_2}\exp\Big[(m-1)\Big(\frac{g(t_l)}{t_l}-\frac{\lambda}{t_l^2}\Big)\Big] \nonumber\\
& = & (1+o(1))\cdot e^{(m-1)\psi(t_0)}\sum_{l\in \Omega_2}\exp\Big[\frac{1}{2}\psi''(t_0)\Big(\frac{m-1}{\sqrt{ml}}-t_0\Big)^2m\Big].
\eea
Set $a_m=\gamma m -\sqrt{m}\log m$, $b_m=\gamma m +\sqrt{m}\log m$, $c_m=(m-1)/\sqrt{m}$ and
\bea\lbl{definition_pizza}
\rho(x)=\exp\Big[\frac{1}{2}\psi''(t_0)\big(\frac{c_m}{\sqrt{x}}-t_0\big)^2m\Big]
\eea
for $x>0$. It is easy to check that there exists an absolute constant $C_1>0$ such that
\bea\lbl{drone}
\rho(x)\leq e^{-C_1(\log m)^2}
\eea
 for all $x\in (a_m, b_m)\backslash([a_m]+2, [b_m]-2)$. Hence
\bea\lbl{lum3}
\int_{a_m}^{b_m}\rho(x)\,dx
&=&\Big(\sum_{l=[a_m]}^{[b_m]-1}\int_{l}^{l+1}\rho(x)\,dx\Big) + \epsilon_m \lbl{lum2}
\eea
where $|\epsilon_m| \leq e^{-C_1(\log m)^2}$ for large $m$. By the expression $\rho(x)=\exp\big[\frac{1}{2}\psi''(t_0)\big(\frac{c_m}{\sqrt{x}}-t_0\big)^2m\big]$, we get
\beaa
\rho'(x)=-\frac{1}{2}\rho(x)\psi''(t_0)\Big(\frac{c_m}{\sqrt{x}}-t_0\Big)\frac{m {c_m}}{x^{3/2}}
\eeaa
for $x>0$. Easily, $\frac{m {c_m}}{x^{3/2}}=O(1)$ and $\frac{c_m}{\sqrt{x}}-t_0=O(\frac{\log m}{\sqrt{m}})$ uniformly for all $[a_m]\leq x\leq [b_m].$ Thus,
\beaa
|\rho'(x)|\leq \frac{(\log m)^2}{\sqrt{m}}\rho(x)
\eeaa
for all $[a_m]\leq x\leq [b_m] $. Therefore, by integration by parts,
\beaa
 \Big|\int_{l}^{l+1}\rho(x)\,dx-\rho(l)\Big| = \Big|\int_{l}^{l+1}\rho'(x)(l+1-x)\,dx\Big|
& \leq &   \int_{l}^{l+1}|\rho'(x)|\,dx\\
&\leq & \frac{(\log m)^2}{\sqrt{m}}\int_{l}^{l+1}\rho(x)\,dx
\eeaa
as $m$ is sufficiently large. This, (\ref{drone}) and (\ref{lum3}) imply
\bea\lbl{snow_heat}
& & \Big|\sum_{l\in \Omega_2}\rho(l)-\int_{a_m}^{b_m}\rho(x)\,dx\Big| \leq \frac{(\log m)^2}{\sqrt{m}}\Big(\int_{a_m}^{b_m}\rho(x)\,dx\Big) + e^{-C_1(\log m)^2}.
\eea
Set $\gamma_m=(\log m)\gamma^{-3/2}/2$. We see from (\ref{peach1}) and (\ref{definition_pizza}) that
\beaa
\int_{a_m}^{b_m}\rho(x)\,dx &= & \frac{2c_m^2}{\sqrt{m}}\int_{-\gamma_m +o(1)}^{\gamma_m +o(1)}\Big(-\frac{u}{\sqrt{m}}+t_0\Big)^{-3}e^{\frac{1}{2}\psi''(t_0)u^2}\,du\\
& = & (1+o(1))\frac{2\sqrt{m}}{t_0^3}\int_{-\gamma_m}^{\gamma_m}e^{\frac{1}{2}\psi''(t_0)u^2}\,du\\
& = & (1+o(1))\frac{2\sqrt{m}}{t_0^3}\int_{-\infty}^{\infty}e^{\frac{1}{2}\psi''(t_0)u^2}\,du\\
& \sim & \sqrt{m}\cdot \frac{1}{t_0^3}\sqrt{\frac{8\pi}{|\psi''(t_0)|}}
\eeaa
by making the transform $u=-\big(\frac{c_m}{\sqrt{x}}-t_0\big)\sqrt{m}$. Combining this, (\ref{peach1}) and (\ref{snow_heat}), we arrive at
\bea\lbl{peach8}
 e^{-(m-1)\psi(t_0)}\sum_{l\in \Omega_2}\exp\Big[(m-1)\Big(\frac{g(t_l)}{t_l}-\frac{\lambda}{t_l^2}\Big)\Big]
& = & (1+o(1))\sum_{l\in \Omega_2}\rho(l)\nonumber\\
& \sim &   \sqrt{m}\cdot \frac{1}{t_0^3}\sqrt{\frac{8\pi}{|\psi''(t_0)|}}
\eea
as $n$ is sufficiently large.   This and (\ref{god_amen}) yield
\bea
& & \sum_{l\in \Omega_2}q^{l } \cdot |\ml{P}_{lm-j}(m-1)| \nonumber\\
& \sim & \frac{t_0^2f(t_0)}{m^2}e^{\lambda t_0^{-2}-(\lambda j/m)}\sum_{l\in \Omega_2}\exp\Big\{(m-1)\big(\frac{g(t_l)}{t_l}-\frac{\lambda}{t_l^2}\big)\Big\} \nonumber\\
&  \sim  & \frac{f(t_0)e^{\lambda t_0^{-2}-\psi(t_0)-(\lambda j/m)}}{t_0}\cdot \sqrt{\frac{8\pi}{|\psi''(t_0)|}}\cdot \frac{e^{m\psi(t_0)}}{m^{3/2}} \lbl{red_gross}
\eea
as $m\to\infty$.

{\it Step 4. Wrap-up of the denominator}. By the choice of $c$ in \eqref{littlec}, we have $\sqrt{c}\le (4K)^{-1}\psi(t_0)$ in (\ref{brain_head}). Therefore we get from (\ref{pig}) that
\bea\lbl{ding}
\Big(\sum_{l=1}^{cm}+\sum_{l=Cm}^{M_n+1}\Big)q^{l } \cdot |\ml{P}_{lm-j}(m-1)| \leq  e^{\psi(t_0)m/2}
\eea
as $n$ is large. This and (\ref{hayes}) imply
\beaa
\sum_{l=1}^{M_n+1}q^{l } \cdot |\ml{P}_{lm}(m-1)| = O\big(e^{\psi(t_0)m/2}\big)+\sum_{i=1}^3\sum_{l\in \Omega_i}q^{l } \cdot |\ml{P}_{lm-j}(m-1)|
\eeaa
as $m\to\infty.$ This identity together with  (\ref{air_exit}) and (\ref{red_gross}) concludes that
\bea\lbl{victory}
 \sum_{l=1}^{M_n+1}q^{l } \cdot |\ml{P}_{lm-j}(m-1)|
\, \sim \,  \frac{f(t_0)e^{\lambda t_0^{-2}-\psi(t_0)-(\lambda j/m)}}{t_0}\cdot \sqrt{\frac{8\pi}{|\psi''(t_0)|}}\cdot \frac{e^{m\psi(t_0)}}{m^{3/2}}
\eea
as $m\to\infty$.

{\it Step 5. Numerator}. We need to show
\beaa
\lim_{n\to\infty}P\Big(\frac{1}{\sqrt{m}}\Big(k_1  -\lceil \frac{n}{m}\rceil - \frac{m}{t_0^2}\Big) \leq x\Big)= \frac{1}{\sqrt{2\pi}\, \sigma}\int_{-\infty}^xe^{-\frac{t^2}{2\sigma^2}}\,dt
\eeaa
for every $x \in \mathbb{R}$, where $\sigma=\frac{1}{\sqrt{|\psi''(t_0)|}}$. Recall $\gamma=t_0^{-2}$. By (\ref{aha}),
\bea\lbl{concrete_leave}
P\Big(\frac{1}{\sqrt{m}}\Big(k_1  -\lceil \frac{n}{m}\rceil - \frac{m}{t_0^2}\Big) \leq x\Big)=
\frac{\sum_{l=1}^{b_m'} q^{l} \cdot |\ml{P}_{ ml-j}(m-1)|}{\sum_{l=1}^{M_n+1} q^{l } \cdot |\ml{P}_{lm-j}(m-1)|}
\eea
where $b_m'=[\gamma m + \sqrt{m}\,x]+1$. Recall $\sqrt{c}\le (4K)^{-1}\psi(t_0)$ be as before. It is known from (\ref{ding}) that
\bea\lbl{quiet_water}
\sum_{l=1}^{cm}q^{l } \cdot |\ml{P}_{lm-j}(m-1)| \leq  e^{\psi(t_0)m/2}
\eea
as $n$ is large. Let $\Omega_1$ and $\Omega_2$ be as in (\ref{soil_white}). Set $\Omega_2'=\{l\in \mathbb{N};\, \gamma m -\sqrt{m}\log m\leq l \leq b_m'\}.$ Notice $\Omega_2'\subset \Omega_2$ for large $m$. By (\ref{air_exit}), (\ref{god_amen}) and (\ref{quiet_water}),
\bea
& & \sum_{l=1}^{b_m'} q^{l} \cdot |\ml{P}_{ ml-j}(m-1)| \nonumber\\
&= & O\big(e^{\psi(t_0)m/2}\big)+ \sqrt{m}e^{m\psi(t_0)-(L/2)(\log m)^2}+\sum_{l\in  \Omega_2'}q^{l} \cdot |\ml{P}_{ ml-j}(m-1)| \nonumber\\
& = & O\Big(\sqrt{m}\cdot e^{m\psi(t_0)-(L/2)(\log m)^2}\Big)+ \nonumber\\
& &~~~~~~~~~~~~~~~~\frac{t_0^2f(t_0)}{m^2}e^{\lambda t_0^{-2}-(\lambda j/m)}\sum_{l\in \Omega'_2}\exp\Big\{(m-1)\Big(\frac{g(t_l)}{t_l}-\frac{\lambda}{t_l^2}\Big)\Big\} \lbl{pecan_eat}
\eea
as $m\to \infty$. Review the derivation between (\ref{peach1}) and (\ref{peach8}) and replace $b_m$ by $b_m'$. by the fact $\Omega_2'\subset \Omega_2$ for large $m$ again, we have
\beaa
& & e^{-(m-1)\psi(t_0)}\sum_{l\in \Omega_2'}\exp\Big[(m-1)\Big(\frac{g(t_l)}{t_l}-\frac{\lambda}{t_l^2}\Big)\Big] \nonumber\\
& =& \int_{a_m}^{b_m'}\rho(x)\,dx + \epsilon_m+ O\big(\frac{\log m}{m^{1/4}}\big)
\eeaa
where, as mentioned before,  $a_m=\gamma m -\sqrt{m}\log m$ and $|\epsilon_m| \leq e^{-C_1(\log m)^2}$ for large $m$. Let us evaluate the integral above. In fact, from (\ref{definition_pizza})  we see that
\beaa
\int_{a_m}^{b_m'}\rho(x)\,dx=\int_{a_m}^{b_m'}\exp\Big[\frac{1}{2}\psi''(t_0)\big(\frac{c_m}{\sqrt{x}}-t_0\big)^2m\Big]\,dx.
\eeaa
Set $w=-\big(\frac{c_m}{\sqrt{x}}-t_0\big)\sqrt{m}$. Then
\beaa
\int_{a_m}^{b_m'}\rho(x)\,dx & = & \frac{2c_m^2}{\sqrt{m}}\int_{-\gamma_m +o(1)}^{\frac{x}{2\gamma^{3/2}}+o(1)}\Big(-\frac{w}{\sqrt{m}}+t_0\Big)^{-3}e^{-\frac{1}{2}|\psi''(t_0)|w^2}\,dw\\
& = & (1+o(1)) \frac{2\sqrt{m}}{t_0^3}\int_{-\infty}^{\frac{x}{2\gamma^{3/2}}}e^{-\frac{1}{2}|\psi''(t_0)|w^2}\,dw\\
& = & (1+o(1)) \frac{\sqrt{m}}{t_0^3\gamma^{3/2}}\int_{-\infty}^{x}e^{-w^2/(2\sigma^2)}\,dw =  (1+o(1))\sqrt{m} \int_{-\infty}^{x}e^{-w^2/(2\sigma^2)}\,dw
\eeaa
where $\gamma_m=(\log m)\gamma^{-3/2}/2$ and $\sigma^2=\frac{4\gamma^3}{|\psi''(t_0)|}.$ Collect the assertions from (\ref{pecan_eat}) to the above to obtain
\beaa
& & \sum_{l=1}^{b_m'} q^{l} \cdot |\ml{P}_{ ml-j}(m-1)|\\
&=& (1+o(1))\frac{t_0^2f(t_0)}{m^2}e^{\lambda t_0^{-2}-(\lambda j/m)}\cdot e^{(m-1)\psi(t_0)}\cdot \sqrt{m} \int_{-\infty}^{x}e^{-w^2/(2\sigma^2)}\,dw\\
&\sim & t_0^{2}f(t_0)\cdot e^{\lambda t_0^{-2}-\psi(t_0)-(\lambda j/m)}\cdot \frac{e^{m\psi(t_0)}}{m^{3/2}}\int_{-\infty}^{x}e^{-w^2/(2\sigma^2)}\,dw
\eeaa
as $m\to\infty$. Join this with (\ref{victory}) and (\ref{concrete_leave}) to conclude that
\bea
P\Big(\frac{1}{\sqrt{m}}\Big(k_1  -\lceil \frac{n}{m}\rceil - \frac{m}{t_0^2}\Big) \leq x\Big)\to \frac{1}{\sqrt{2\pi}\, \sigma}\int_{-\infty}^{x}e^{-w^2/(2\sigma^2)}\,dw
\eea
as $m\to\infty$. Notice that $\sigma^2=\frac{4}{|\psi''(t_0)|t_0^6}$. The proof is completed by using Lemma \ref{London} and the fact   $\gamma=t_0^{-2}$.
\end{proof}

\section{Proofs of Theorems \ref{thm:general} and \ref{thm:general-infinity} and Corollaries \ref{cor:dir} and \ref{cor:sphere}}\label{sec:general}

In Section \ref{sec:generalfix} below, we will prove Theorem \ref{thm:general}, Corollaries \ref{cor:dir} and \ref{cor:sphere} where $m$ is assumed to be a fixed integer. Theorem \ref{thm:general-infinity} studies the case when $m$ tends to infinity with $n$ and $m=o(n^{1/3})$. Its proof is given in Section \ref{sec:generalinf}.

\subsection{The Proofs of Theorem \ref{thm:general} and Corollaries \ref{cor:dir} and \ref{cor:sphere}}\lbl{sec:generalfix}

From \citet{EL}, we have
\bea\lbl{eq:size}
\ml{P}_n(m) \sim \frac{\binom{n-1}{m-1}}{m!}
\eea
uniformly for $m=o(n^{1/3})$.
\begin{proof}[Proof of Theorem \ref{thm:general}]
To prove the conclusion, it suffices to show that for any bounded continuous function $\psi$ on $\overline{\nabla}_{m-1}$,
$$\E\left(\psi(\frac{k_1}{n},\ldots,\frac{k_m}{n}) \right) \to \E\left(\psi(x_1,\ldots,x_m) \right)$$
as $n$ tends to infinity, where $(x_1,\ldots,x_m) \sim \mu$.
By definition,
\bea\lbl{eq:part}
\E\left(\psi(\frac{k_1}{n},\ldots,\frac{k_m}{n}) \right)
&=& \frac{\sum_{(k_1,\ldots,k_m) \in \ml{P}_n(m)} \psi(\frac{k_1}{n},\ldots,\frac{k_m}{n}) f(\frac{k_1}{n},\ldots,\frac{k_m}{n})}{\sum_{(k_1,\ldots,k_m) \in \ml{P}_n(m)} f(\frac{k_1}{n},\ldots,\frac{k_m}{n})}\\ \nonumber
&=&  \frac{ n^{-(m-1)} \sum_{(k_1,\ldots,k_m) \in \mathcal{R}_n(m)}\psi(\frac{k_1}{n},\ldots,\frac{k_m}{n}) f(\frac{k_1}{n},\ldots,\frac{k_m}{n})}{n^{-(m-1)} \sum_{(k_1,\ldots,k_m) \in \ml{P}_n(m)} f(\frac{k_1}{n},\ldots,\frac{k_m}{n})}+\mathcal{E}_{n,m},
\eea
where the set
\beaa
\mathcal{R}_n(m):= \{(k_1,\ldots,k_m) \vdash n; k_1 > \ldots >k_m >0\}
\eeaa
and
\beaa
\mathcal{E}_{n,m}:=\frac{\sum_{(k_1,\ldots,k_m) \in \ml{P}_n(m)\setminus \mathcal{R}_n(m)} \psi(\frac{k_1}{n},\ldots,\frac{k_m}{n}) f(\frac{k_1}{n},\ldots,\frac{k_m}{n}) }{\sum_{(k_1,\ldots,k_m) \in \ml{P}_n(m)} f(\frac{k_1}{n},\ldots,\frac{k_m}{n})}.
\eeaa
On the other hand,
\bea\lbl{eq:gen}
\E(\psi(x_1,\ldots,x_m)) &=& \int_{\nabla_{m-1}} \psi(y_1,\ldots,y_m) f(y_1,\ldots,y_m) \,dy_1\ldots dy_{m-1}\\ \nonumber
&=& \frac{\int_{\nabla_{m-1}} \psi(y_1,\ldots,y_m) f(y_1,\ldots,y_m) \,dy_1\ldots dy_{m-1}}{\int_{\nabla_{m-1}} f(y_1,\ldots,y_m) \,dy_1\ldots dy_{m-1}}.
\eea
In order to compare \eqref{eq:part} and \eqref{eq:gen}, we divide the proof into a few steps.

{\it Step 1: Estimate of $|\mathcal{E}_{n,m}|$}. We claim that the term $\mathcal{E}_{n,m}$ is negligible as $n\to \infty$. We first estimate the size of $\mathcal{R}_n(m)$.
For any $(k_1, \cdots, k_m)\in \mathcal{R}_n(m) $, set $j_i = k_i-(m-i+1)$ for $1\le i \le m$. It is easy to verify that $j_{i-1} - j_i =k_{i-1}-k_{i}-1\ge 0$ for $2\le i \le m$. Thus $$j_1 + \cdots + j_{m} = n-\binom{m+1}{2}$$ and $j_1 \ge \cdots \ge j_m \ge 0$. Therefore,
$(j_1,\cdots,j_m) \in \ml{P}_{n-\binom{m+1}{2}}(m).$ Indeed, this transform is a bijection between $\mathcal{R}_n(m)$ and $\ml{P}_{n-\binom{m+1}{2}}(m)$, which implies $$|\mathcal{R}_n(m)|=|\ml{P}_{n-\binom{m+1}{2}}(m)|.$$
On the other hand, we know from \eqref{eq:size},
$$|\ml{P}_N(m)| \sim \frac{\binom{N-1}{m-1}}{m!}$$
as $N\to \infty$.
Thus by Stirling's formula,
\beaa
\frac{|\mathcal{R}_n(m)|}{|\ml{P}_n(m)|} &\sim& \frac{\binom{n-\binom{m+1}{2}-1}{m-1} }{\binom{n-1}{m-1}}=\frac{(n-\binom{m+1}{2}-1)!(n-m)!}{(n-1)!(n-\binom{m+1}{2}-m)!} \\
&\sim& \frac{(n-\binom{m+1}{2})!(n-m)!}{n!(n-\binom{m+1}{2}-m)!}\\
&\sim& \frac{(1-\frac{m}{n})^{1/2}}{(1-\frac{m}{n-\binom{m+1}{2}})^{1/2}}\frac{(1-\frac{m}{n})^n (1-\frac{\binom{m+1}{2}}{n-m})^m}{(1-\frac{m}{n-\binom{m+1}{2}})^{n-\binom{m+1}{2}}}
\eeaa
as $n\to \infty$. By assumption $m=o(\sqrt{n})$, we have $\frac{n-\binom{m+1}{2}}{m} \to \infty$ with $n$. Using the fact that $\lim_{N\to \infty} (1+\frac{x}{N})^N = \exp(x)$, we obtain
\beaa
\frac{|\mathcal{R}_n(m)|}{|\ml{P}_n(m)|} \sim \exp\left(-\frac{m\binom{m+1}{2}}{n-m} \right).
\eeaa
Thus as long as $m = o(n^{1/3})$,
\beaa
 |\mathcal{R}_n(m)| \sim |\ml{P}_n(m)|\ \ \ \ \mbox{and}\ \ \ \
 |\ml{P}_n(m)\setminus \mathcal{R}_n(m)| = o(|\ml{P}_n(m)|)
\eeaa
as $n\to \infty$.

 Further, since $\int_{\nabla_{m-1}} f(y_1,\ldots,y_m) \, dy_1\ldots dy_{m-1}=1$, there exists a region $\mathcal{S}$ on $\overline{\nabla}_{m-1}$
 whose measure $|\mathcal{S}| \ge \mu |\nabla_{m-1}|$ for some constant $\mu>0$ such that $f(y_1,\ldots,y_m)>c$ on $\mathcal{S}$ for some $c>0$. Thus, for $n$ sufficiently large, $f(k_1/n,\ldots,k_m/n)> c_0>0$ for $(k_1,\ldots,k_m)$ in a subset of  $\ml{P}_n(m)$ with cardinality at least a small fraction of $|\ml{P}_n(m)|$. Also since the functions $\psi$ and $f$ are bounded on  $\nabla_{m-1}$, we conclude
\bea\lbl{eq:error}
|\mathcal{E}_{n,m}| = O\left(\frac{|\ml{P}_n(m)\setminus \mathcal{R}_n(m)|}{|\ml{P}_n(m)|}\right) = o(1)
\eea
as $n\to \infty$, as long as $m = o(n^{1/3})$.

{\it Step 2: Compare the numerators of \eqref{eq:part} and \eqref{eq:gen}}. For convenience, denote
\begin{equation}\label{def:G}
G(y_1,\ldots,y_{m-1}) = \psi\Big(y_1,\ldots,y_{m-1},1-\sum_{i=1}^{m-1}y_i\Big) f\Big(y_1,\ldots,y_{m-1},1-\sum_{i=1}^{m-1}y_i\Big).
\end{equation}
Since $\psi, f$ are bounded continuous functions on $\overline{\nabla}_{m-1}$, it is easy to check that $G$ is also bounded and continuous on $\overline{\nabla}_{m-1}$.
We can rewirte the numberator in \eqref{eq:part} as follows.
\beaa
\mathcal{I}_1 &:=& \frac{1}{n^{m-1}} \sum_{\substack{k_1 > \ldots >k_m >0 \\ k_1+\ldots+k_m=n}} G\Big(\frac{k_1}{n},\ldots,\frac{k_{m-1}}{n}\Big) \\
&=& \frac{1}{n^{m-1}} \sum_{(k_1,\ldots,k_{m-1}) \in \{1,\ldots,n\}^{m-1}} G\Big(\frac{k_1}{n},\ldots,\frac{k_{m-1}}{n}\Big) I_{\mathcal{A}_n}\\
&=& \sum_{(k_1,\ldots,k_{m-1}) \in \{1,\ldots,n\}^{m-1}} \int_{\frac{k_1-1}{n}}^{\frac{k_1}{n}} \cdots \int_{\frac{k_{m-1}-1}{n}}^{\frac{k_{m-1}}{n}} G\Big(\frac{k_1}{n},\ldots,\frac{k_{m-1}}{n}\Big) I_{ \mathcal{A}_n } ~d y_1 \dots d y_{m-1},
\eeaa
where $I_{ \mathcal{A}_n }$ is the indicator function of set  $\mathcal{A}_n$ defined as below
\bea\lbl{eq:An}
\mathcal{A}_n= \frac{1}{n} \Big\{ (k_1, \cdots, k_{m-1})\in \{1,\ldots,n \}^{m-1};\, \frac{k_1}{n} > \cdots > \frac{k_{m-1}}{n} > 1- \sum_{i=1}^{m-1} \frac{k_i}{n} > 0\Big\}.
\eea
Similarly,
\beaa
\mathcal{I}_2 &:=& \int_{\nabla_{m-1}} G(y_1,\ldots,y_{m-1})  \,dy_1\ldots dy_{m-1} \\
&=& \int_{[0,1]^{m-1}} G(y_1,\ldots,y_{m-1})I_{\mathcal{A}}  \,dy_1\ldots dy_{m-1}\\
&=& \sum_{(k_1,\ldots,k_{m-1}) \in \{1,\ldots,n\}^{m-1}} \int_{\frac{k_1-1}{n}}^{\frac{k_1}{n}} \cdots \int_{\frac{k_{m-1}-1}{n}}^{\frac{k_{m-1}}{n}} G(y_1,\ldots,y_{m-1}) I_{ \mathcal{A} } ~d y_1 \dots d y_{m-1},
\eeaa
where the $I_{ \mathcal{A}}$ is the indicator function of set  $\mathcal{A}$  denoted by
\bea\lbl{eq:A}
 \mathcal{A}=\Big\{(x_1, \cdots, x_{m-1})\in [0,1]^{m-1};\, x_1 > \cdots > x_{m-1} > 1- \sum_{i=1}^{m-1} x_i \ge 0\Big\}.
\eea
Now we estimate the difference between the numerators in \eqref{eq:part} and \eqref{eq:gen}.
\beaa
&& \mathcal{I}_1 - \mathcal{I}_2\\
& = & \sum_{(k_1,\ldots,k_{m-1}) \in \{1,\ldots,n\}^{m-1}} \int_{\frac{k_1-1}{n}}^{\frac{k_1}{n}} \cdots \int_{\frac{k_{m-1}-1}{n}}^{\frac{k_{m-1}}{n}} \nonumber\\
& & \quad \quad \quad \quad \quad \quad  \quad \quad \left( G\Big(\frac{k_1}{n},\ldots,\frac{k_{m-1}}{n}\Big) I_{ \mathcal{A}_n } - G(y_1,\ldots,y_{m-1}) I_{ \mathcal{A} } \right) ~d y_1 \dots d y_{m-1}
\eeaa
which is identical to
\beaa
&&\sum_{(k_1,\ldots,k_{m-1}) \in \{1,\ldots,n\}^{m-1}} \int_{\frac{k_1-1}{n}}^{\frac{k_1}{n}} \cdots \int_{\frac{k_{m-1}-1}{n}}^{\frac{k_{m-1}}{n}} \nonumber\\
& & \quad \quad \quad \quad \quad \quad  \quad  \left(G\Big(\frac{k_1}{n},\ldots,\frac{k_{m-1}}{n}\Big) - G(y_1,\ldots,y_{m-1})\right) I_{ \mathcal{A}_n }  ~d y_1 \dots d y_{m-1}\\
& & \quad \quad + \sum_{(k_1,\ldots,k_{m-1}) \in \{1,\ldots,n\}^{m-1}} \int_{\frac{k_1-1}{n}}^{\frac{k_1}{n}} \cdots \int_{\frac{k_{m-1}-1}{n}}^{\frac{k_{m-1}}{n}} \nonumber\\
& & \quad \quad \quad \quad \quad \quad  \quad \quad G(y_1,\ldots,y_{m-1}) \left( I_{ \mathcal{A}_n } - I_{ \mathcal{A} }  \right)~d y_1 \dots d y_{m-1}\\
&& := \mathcal{S}_1 + \mathcal{S}_2.
\eeaa

{\it Step 3: Estimate $\mathcal{S}_1$.} Since $G$ is uniformly continuous on $\overline{\nabla}_{m-1}$, for any $\varepsilon>0$ and any $y_i \in [\frac{k_i-1}{n}, \frac{k_i}{n}]~(1\le i \le m-1)$,
\begin{align}\label{eq:bdG}
 \Big|G\Big(\frac{k_1}{n},\ldots,\frac{k_{m-1}}{n}\Big) - G(y_1,\ldots,y_{m-1})\Big| < \varepsilon,
\end{align}
when $n$ is sufficiently large.
Thus
\begin{align}\label{eq:S1}
|\mathcal{S}_1| &\le \sum_{(k_1,\ldots,k_{m-1}) \in \{1,\ldots,n\}^{m-1}} \int_{\frac{k_1-1}{n}}^{\frac{k_1}{n}} \cdots \int_{\frac{k_{m-1}-1}{n}}^{\frac{k_{m-1}}{n}} \nonumber\\
&  \quad \quad \quad \quad \quad \quad  \quad  \Big|G\Big(\frac{k_1}{n},\ldots,\frac{k_{m-1}}{n}\Big) - G(y_1,\ldots,y_{m-1})\Big|  ~d y_1 \dots d y_{m-1} \nonumber\\
&\le \varepsilon  \Big(\frac{1}{n}\Big)^{m-1} n^{m-1}\nonumber\\
& = \varepsilon
\end{align}
for $n$ sufficiently large.

{\it Step 4: Estimate $\mathcal{S}_2$.} Since $G$ is bounded on $\overline{\nabla}_{m-1}$,  $\|G\|_{\infty} :=\sup_{\bf{x}\in \overline{\nabla}_{m-1}} |G(\bf{x})| < \infty$ and thus
\bea\lbl{eq:S2}
|\mathcal{S}_2| \le \|G\|_{\infty}\sum_{(k_1,\ldots,k_{m-1}) \in \{1,\ldots,n\}^{m-1}} \int_{\frac{k_1-1}{n}}^{\frac{k_1}{n}} \cdots \int_{\frac{k_{m-1}-1}{n}}^{\frac{k_{m-1}}{n}} | I_{ \mathcal{A}_n } - I_{ \mathcal{A}}|~d y_1 \dots d y_{m-1}.
\eea
Now we control $| I_{ \mathcal{A}_n } - I_{ \mathcal{A}}|$ provided $\frac{k_i -1}{n} < y_i < \frac{k_i}{n}$ for $1\le i \le m-1$.
By definition,
\bea\lbl{eq:id-An}
I_{ \mathcal{A}_n}= \left\{
     \begin{array}{lr}
       1 ,\ \text{if}~\frac{k_1}{n} > \cdots > \frac{k_{m-1}}{n} > 1- \sum_{i=1}^{m-1} \frac{k_i}{n} > 0\\
       0 ,\ \text{otherwise}
     \end{array}
   \right.
\eea
and
\bea\lbl{eq:id-A}
I_{ \mathcal{A}}= \left\{
     \begin{array}{lr}
       1 ,\ \text{if}~y_1 > \cdots > y_{m-1} > 1- \sum_{i=1}^{m-1} y_i \ge 0\\
       0 ,\ \text{otherwise.}
     \end{array}
   \right.
\eea
Let $\mathcal{B}_n$ be a subset of $\mathcal{A}_n$ such that
\beaa
\mathcal{B}_n= \mathcal{A}_n \cap \Big\{(k_1, \cdots, k_{m-1})\in \{1,\ldots,n\}^{m-1};\, \frac{k_{m-1}}{n} +\sum_{i=1}^{m-1} \frac{k_i}{n} > \frac{m}{n}+1\Big\}.
\eeaa
 Given $(k_1, \cdots, k_{m-1})\in \mathcal{B}_n$, for any
\bea\lbl{red_book}
\frac{k_1 - 1}{n} < y_1 < \frac{k_1}{n}, \cdots, \frac{k_{m-1} - 1}{n} < y_{m-1} < \frac{k_{m-1}}{n},
\eea
it is easy to verify from (\ref{eq:id-A}) and (\ref{eq:id-An}) that $I_{\mathcal{A}}=1$. Hence,
 \bea
 I_{\mathcal{A}_n}&=&I_{\mathcal{B}_n} + I_{\mathcal{A}_n\backslash\mathcal{B}_n} \nonumber\\
& \leq & I_{\mathcal{A}} + I_{\mathcal{A}_n \cap \{ k_{m-1} + \sum_{i=1}^{m-1}k_i \leq  n+m \}} \nonumber\\
& = & I_{\mathcal{A}} + \sum_{j=n+1}^{n+m}I_{E_j}\lbl{spring_cold}
 \eea
where
\beaa
E_j&=&\Big\{(k_1, \cdots, k_{m-1})\in \{1,\ldots,n\}^{m-1};\, k_1>\ldots>k_{m-1}\ge 1, \\
& & \quad \quad \quad \quad \quad \quad\quad\quad\quad\quad\quad\quad\quad k_{m-1} + \sum_{i=1}^{m-1}k_i=j, \sum_{i=1}^{m-1} k_i <n\Big\}
\eeaa
for $n+1 \leq j \leq m+n$. Let us estimate the size of $|E_j|$. From the last two restrictions, we obtain $k_{m-1} >j-n$. Since $\sum_{i=1}^{m-1} k_i <n$ and $k_i > k_{m-1}$ for $1\le i \le m-2$, we have $j-n+1 \le k_{m-1} \le \frac{n}{m-1}$.

For each fixed $k_{m-1}$, since $k_1> \ldots >k_{m-2}$ is the ordered positive integer solution to the linear equation $\sum_{i=1}^{m-2}k_i = j - 2k_{m-1}$, thus
\beaa
|E_{j}| \le \sum_{j-n+1 \le l \le \frac{n}{m-1}}^{} \frac{\binom{j-2l-1}{m-3}}{(m-2)!} \le \left(\frac{n}{m-1}+n-j \right)\frac{\binom{2n-j-3}{m-3}}{(m-2)!}.
\eeaa
As a result, we obtain the crude upper bound
\bea\lbl{eq:E}
\sum_{j=n+1}^{n+m}|E_j| \le \sum_{j=n+1}^{n+m} \left(\frac{n}{m-1}+n-j \right)\frac{\binom{2n-j-3}{m-3}}{(m-2)!}\le  \frac{m\cdot n^{m-2}}{(m-1)!(m-3)!}.
\eea
On the other hand, consider a subset of $\mathcal{A}_n^c:=\{\frac{1}{n},\frac{2}{n}, \cdots, 1\}^{m-1}\backslash \mathcal{A}_n$ defined by
\beaa
\mathcal{C}_n &=&
\frac{1}{n} \Big\{(k_1, \cdots, k_{m-1})\in \{1,\ldots,n\}^{m-1};\, \mbox{either}\  k_i\leq k_{i+1}-1\ \mbox{for some }
  1\leq i \leq m-2,\\
& & \ \mbox{or}\
 k_1+\cdots + k_{m-2} + 2k_{m-1} \leq   n,\ \mbox{or}\ k_1+\cdots  + k_{m-1} \geq  m+n-1 \Big\}.
\eeaa
Set $\mathcal{A}^c=[0,1]^{m-1}\backslash \mathcal{A}$. Given $(\frac{k_1}{n}, \cdots, \frac{k_{m-1}}{n})\in \mathcal{C}_n$, for any $k_i$'s and $y_i$'s satisfying (\ref{red_book}), it is not difficult to check that $I_{\mathcal{A}^c}=1$. Consequently,
\beaa
I_{\mathcal{A}_n^c} &= & I_{\mathcal{C}_n} + I\Big\{(\frac{k_1}{n}, \cdots, \frac{k_{m-1}}{n})\in \mathcal{A}_n^c;\, k_i> k_{i+1}-1\ \mbox{for all }\ 1\leq i \leq m-2,\\
& & ~~~~~~~~~~~~ k_1+\cdots + k_{m-2} + 2k_{m-1} >  n,\ \mbox{and}\ k_1+\cdots  + k_{m-1} <  m+n-1 \Big\}\\
& \leq & I_{\mathcal{A}^c} + I(\mathcal{D}_{n,m,1}) + I(\mathcal{D}_{n,m,2}),
\eeaa
or equivalently,
\bea\lbl{coca_warm}
I_{\mathcal{A}_n}
 \geq  I_{\mathcal{A}} - I(\mathcal{D}_{n,m,1}) -  I(\mathcal{D}_{n,m,2}),
\eea
where
\beaa
& &\mathcal{D}_{n,m,1}=\bigcup_{l=n}^{ n+m-2}\frac{1}{n}\Big\{(k_1, \cdots, k_{m-1})\in \{1,\ldots,n\}^{m-1};\, \sum_{i=1}^{m-1}k_i=l, k_1 \ge \ldots \ge k_{m-1} \Big\};\\
& & \mathcal{D}_{n,m,2}=\bigcup_{l=1}^{m-2}\frac{1}{n}\Big\{(k_1, \cdots, k_{m-1})\in \{1,\ldots,n\}^{m-1};\, k_l=k_{l+1}, k_1 \ge \ldots \ge k_{m-1},\\
& & \quad\quad\quad\quad\quad\quad\quad\quad\quad\quad \sum_{i=1}^{m-1}k_i + k_{m-1} \ge n+1, \sum_{i=1}^{m-1} k_i \le n+m-2\Big\}.\\
\eeaa
By the definition of partitions and \eqref{eq:size}, we have the following bound on $|\mathcal{D}_{n,m,1}|$.
\bea\lbl{eq:D1}
|\mathcal{D}_{n,m,1}| &\le& \sum_{l=n}^{n+m-2}|\ml{P}_l(m-1)| \sim \sum_{l=n}^{n+m-2} \frac{\binom{l-1}{m-2}}{(m-1)!}\nonumber\\
&\le & (m-1)\frac{\binom{n+m-2}{m-2}}{(m-1)!} \le \frac{(n+m-2)^{m-2}}{[(m-2)!]^2}
\eea
as $n\to \infty$.

The estimation of $|\mathcal{D}_{n,m,2}|$ is the same argument as in (\ref{eq:E}). For the cases $m=3$ or $m=4$, it is easy to verify that $|\mathcal{D}_{n,m,2}|=O(n^{m-2})$. Now we assume $m\ge 5$. First, from the decreasing order of $k_i$ and $\sum_{i=1}^{m-1}k_i \le n+m-2$,
we determine the range of $k_{m-1}$, $$1\le k_{m-1} \le  \frac{n+m-2}{m-1}.$$ On the other hand, $n+1-2k_{m-1} \le \sum_{i=1}^{m-2}k_i \le n+m-2-k_{m-1}$. If $l\neq m-2$, from the restriction $k_l = k_{l+1}$, we see $k_1 + \ldots+k_{l-1} + k_{l+2}+\ldots +k_{m-2} = s-2k_l$ is
the ordered positive integer solutions to the equation $j_1+\ldots+j_{m-4}=s-2k_l$, where $n+1-2k_{m-1} \le s \le n+m-2-k_{m-1}$. If $l=m-2$, then $k_1 + \cdots +k_{m-3} = s-2k_{m-1}$ and $n+1-3k_{m-1} \le s-2k_{m-1} \le n+m-2-2k_{m-1}$. Therefore, we have the following crude upper bound
\bea\lbl{eq:D2}
|\mathcal{D}_{n,m,2}| &\le& \sum_{l=1}^{m-3}\sum_{k_{m-1}=1}^{\frac{n+m-2}{m-1}} \sum_{s=n+1-2k_{m-1}}^{n+m-2-k_{m-1}}\sum_{k_{m-1}\le k_l \le s/2} \frac{\binom{s-2k_l-1}{m-5}}{(m-4)!}\nonumber\\
&& \quad\quad\quad \quad + \sum_{k_{m-1}=1}^{\frac{n+m-2}{m-1}} \sum_{s=n+1-3k_{m-1}}^{n+m-2-2k_{m-1}} \frac{\binom{s-k_{m-1}-1}{m-4}}{(m-3)!}\nonumber\\
&=& O\left( \frac{n^3(m-3)}{m^2(m-4)!} \binom{n+m-6}{m-5} + \frac{n^2}{m^2(m-3)!} \binom{n+m-6}{m-4}\right)\nonumber\\
&=& O\left( \frac{n^2 (n+m)^{m-4} }{m (m-4)!(m-5)!}\right).
\eea
Joining (\ref{spring_cold}) and (\ref{coca_warm}), and assuming (\ref{red_book}) holds, we arrive at
\beaa
|I_{\mathcal{A}_n}-  I_{\mathcal{A}}|\leq  I(\mathcal{D}_{n,m,1})+ I(\mathcal{D}_{n,m,2})  + \sum_{i=n+1}^{n+m}I_{E_i}.
\eeaa
Observe that $\mathcal{D}_{n,m,i}$'s and $E_i$'s do not depend on $y_i$'s, we obtain from (\ref{eq:S2}) that
\beaa
 |\mathcal{S}_2| &\leq & \|G\|_{\infty}\sum_{k_1 = 1}^n \cdots \sum_{k_{m-1}=1}^n \Big[\sum_{i=1}^2I(\mathcal{D}_{n,m,i}) + \sum_{i=n}^{n+m}I_{E_i}\Big]\int_{\frac{k_1-1}{n}}^{\frac{k_1}{n}} \cdots \int_{\frac{k_{m-1}-1}{n}}^{\frac{k_{m-1}}{n}}1 ~d y_1 \dots d y_{m-1}\\
& = &  \|G\|_{\infty} \Big(\sum_{i=1}^2|\mathcal{D}_{n,m,i}| + \sum_{i=n}^{n+m}|E_i|\Big)\cdot \frac{1}{n^{m-1}}.
\eeaa
For $2\le m \le 4$,
$$|\mathcal{S}_2| = O(n^{-1}).$$
For $m\ge 5$, by \eqref{eq:E},\eqref{eq:D1} and \eqref{eq:D2},
\beaa
|\mathcal{S}_2|&=&O\left( \frac{m\cdot n^{m-2}}{(m-1)!(m-3)!} + \frac{(n+m)^{m-2}}{[(m-2)!]^2}+\frac{n^2 (n+m)^{m-4} }{m (m-4)!(m-5)!} \right)\cdot\frac{1}{n^{m-1}}\\
&=& O\left( \frac{(1+\frac{m}{n})^m}{n} \right)
\eeaa
as $n\to\infty.$

{\it Step 5: Difference between the expectations \eqref{eq:part} and \eqref{eq:gen}}. For any $\varepsilon>0$, from {\it Step 3} and {\it Step 4}, we obtain the difference between the numberators in \eqref{eq:part} and \eqref{eq:gen}
\begin{align}\label{eq:diff}
|\mathcal{I}_1 - \mathcal{I}_2| \le |\mathcal{S}_1| + |\mathcal{S}_2| \le \varepsilon + O\left( \frac{(1+\frac{m}{n})^m}{n} \right) < 2 \varepsilon
\end{align}
for $n$ sufficiently large.
Choosing $\psi$ to be identity on $\overline{\nabla}_{m-1}$, we obtain the difference between the denominators in \eqref{eq:part} and \eqref{eq:gen} as follows:
\begin{align*}
\left| n^{-(m-1)} \sum_{(k_1,\ldots,k_m) \in \ml{P}_n(m)} f\Big(\frac{k_1}{n},\ldots,\frac{k_m}{n}\Big) - \int_{\nabla_{m-1}} f(y_1,\ldots,y_m) \,dy_1\ldots dy_{m-1} \right|< 2 \varepsilon
\end{align*}
for $n$ sufficiently large.

Finally, we estimate the expectations \eqref{eq:part} and \eqref{eq:gen}. Since $m$ is fixed, by \eqref{eq:error}, \eqref{eq:diff}, and the triangle inequality,
\begin{align}\label{eq:compare}
\Big|\E\left(\psi\Big(\frac{k_1}{n},\ldots,\frac{k_m}{n}\Big) \right) - \E\left(\psi(x_1,\ldots,x_m) \right)\Big| \to 0
\end{align}
as $n\to \infty$. This completes the proof.
\end{proof}

\begin{proof}[Proof of Corollary \ref{cor:dir}]
By Theorem \ref{thm:general},
$$\Big(\frac{k_1}{n},\ldots, \frac{k_m}{n}\Big) \to (x_1, \ldots, x_m)\sim \mu$$
as $n \to \infty$, where $\mu$ has pdf
\bea\lbl{eq:result}
g(y_1,\ldots, y_m) =\frac{ y_1^{\alpha-1} \cdots y_m^{\alpha-1}}{\int_{\nabla_{m-1}} y_1^{\alpha-1} \cdots y_m^{\alpha-1} \,dy_1\ldots dy_{m-1}}.
\eea
It suffices to show the order statistics $(X_{(1)}, \ldots, X_{(m)})$ of $(X_1,\ldots,X_m) \sim \text{Dir}(\alpha)$ has the same pdf on $\nabla_{m-1}$. For any continuous function $\psi$ defined on $\nabla_{m-1}$, by symmetry,
\beaa
&&\E\psi(X_{(1)}, \ldots, X_{(m)}) \\
&=& \int_{W_{m-1}} \psi(y_{(1)}, \ldots, y_{(m)})\mathbf{1}_{\{y_{(1)}\ge \ldots \ge y_{(m)} \}}\frac{\Gamma(m\alpha)}{\Gamma(\alpha)^m} y_1^{\alpha-1} \cdots y_m^{\alpha-1}\,dy_1\ldots dy_{m-1}\\
&=& \int_{W_{m-1}} \sum_{\sigma\in \mathcal{S}_m} \psi(y_{\sigma(1)}, \ldots, y_{\sigma(m)})\mathbf{1}_{\{y_{\sigma(1)}\ge \ldots \ge y_{\sigma(m)} \}} \frac{\Gamma(m\alpha)}{\Gamma(\alpha)^m} y_{\sigma(1)}^{\alpha-1} \cdots y_{\sigma(m)}^{\alpha-1}\,dy_1\ldots dy_{m-1}\\
&=& \int_{\nabla_{m-1}} \psi(y_1,\ldots,y_m)  \frac{m!\Gamma(m\alpha)}{\Gamma(\alpha)^m} y_{1}^{\alpha-1} \cdots y_{m}^{\alpha-1}\,dy_1\ldots dy_{m-1}.
\eeaa

Therefore, the pdf of $(X_{(1)}, \ldots, X_{(m)})$ is
\bea\lbl{eq:ordered}
\frac{m!\Gamma(m\alpha)}{\Gamma(\alpha)^m} y_{1}^{\alpha-1} \cdots y_{m}^{\alpha-1}
\eea
on the set $\nabla_{m-1}$. Similarly, by the definition of pdf we have $$\int_{W_{m-1}}\frac{\Gamma(m\alpha)}{\Gamma(\alpha)^m} x_1^{\alpha-1} \cdots x_m^{\alpha-1}=1.$$ By symmetry, we obtain
$$\int_{\nabla_{m-1}} y_1^{\alpha-1} \cdots y_m^{\alpha-1} \,dy_1\ldots dy_{m-1}=\frac{\Gamma(\alpha)^m}{m!\Gamma(m\alpha)}.$$
Comparing the above with \eqref{eq:ordered} and \eqref{eq:result}, we complete the proof.
\end{proof}

\begin{proof}[Proof of Corollary \ref{cor:sphere}]
By Theorem \ref{thm:general} or Corollary \ref{cor:dir},
$$\Big(\frac{k_1}{n},\ldots, \frac{k_m}{n}\Big) \to (\widetilde Y_1, \ldots, \widetilde Y_m)\sim \mu$$
as $n \to \infty$, where $\mu$ has pdf
$$\frac{m!\cdot \Gamma(\frac{m}{\alpha})}{\Gamma(\frac{1}{\alpha})^{m}}(y_1\ldots y_{m})^{\alpha-1}
$$
on $\nabla_{m-1}$ and zero elsewhere. Since $f(x)=x^{\alpha}$ is continuous,
$$\left( \Big(\frac{k_1}{n}\Big)^{\alpha},\ldots, \Big(\frac{k_m}{n}\Big)^{\alpha}\right) \to \left( \widetilde Y_1^{\alpha}, \ldots, \widetilde Y_m^{\alpha} \right)$$
as $n \to \infty$.

Now it suffices to show $\left( \widetilde Y_1^{\alpha}, \ldots, \widetilde Y_m^{\alpha} \right)$ has the uniform distribution on the set
$$\mathcal{U}_{m-1}=\Big\{(x_1,\ldots,x_m)\in [0,1]^m;\sum_{i=1}^m x_i^{\frac{1}{\alpha}}=1, x_1\ge \ldots \ge x_m \Big\}.$$
This can be seen by change of variables. For any continuous function $\psi$ defined on $\nabla_{m-1}$,
\beaa
&&\E\psi(\widetilde Y_1^{\alpha}, \ldots,  \widetilde Y_m^{\alpha}) \\
&=& \int_{\nabla_{m-1}} \psi(y_1^{\alpha}, \ldots, y_m^{\alpha})\frac{m!\cdot \Gamma(\frac{m}{\alpha})}{\Gamma(\frac{1}{\alpha})^{m}} y_1^{\alpha-1} \cdots y_m^{\alpha-1}\,dy_1\ldots dy_{m-1}\\
&=& \int_{\mathcal{U}_{m-1}} \psi(x_1, \ldots, x_m)\frac{m!\cdot \Gamma(\frac{m}{\alpha})}{\alpha^{m-1}\Gamma(\frac{1}{\alpha})^{m}} \,dx_1\ldots dx_{m-1}.
\eeaa
In the last equality, we set $x_i=y_i^{\alpha}$ for $1\le i \le m$. Therefore, we can see the pdf of $(\widetilde Y_1^{\alpha}, \ldots,\widetilde Y_m^{\alpha})$ is a constant on $\mathcal{U}_{m-1}$,
which is the uniform distribution on $\mathcal{U}_{m-1}.$ The proof is complete.
\end{proof}

\subsection{The Proof of Theorem \ref{thm:general-infinity}}\lbl{sec:generalinf}
In Section \ref{sec:generalfix} we have studied the asymptotic distribution of $(\frac{k_1}{n}, \cdots, \frac{k_m}{n})$ as $m$ is fixed. Now we consider the case that $m$ depends on $n$. Note that the formula \eqref{eq:size} holds as long as $m=o(n^{1/3})$.

Let $\mu$ and $\nu$ be two Borel
probability measures on a Polish space $S$ with the Borel $\sigma$-algebra $\mathcal{B}(S)$.  Define
\begin{eqnarray}\lbl{cream}
  \rho(\mu, \nu)
  & = & \sup_{\| \varphi \|_L\leq 1}\left|\int_{S} \varphi(x)\, \mu(dx) -
  \int_{S}\varphi(x)\, \nu(dx)\right|,
\end{eqnarray}
where $\varphi$ is a bounded Lipschitz function defined on
$S$ with $\|\varphi\|=\sup_{x\in S}|\varphi(x)|,$ and
$\|\varphi\|_L=\|\varphi\|+\sup_{x\ne y}|\varphi(x)-\varphi(y)|/|x-y|.$ It is known that $\mu_n$ converges to $\mu$ weakly  if and only if $\lim_{n\to\infty}\int \varphi(x)\, \mu_n(dx) = \int \varphi(x)\, \mu(dx)$  for every bounded and Lipschitz continuous function $\varphi(x)$ defined on $\mathbb{R}^m$, and if and
only if $\lim_{n\to\infty}\rho(\mu_n, \mu)=0$; see, e.g., Chapter 11 from \citet{Dudley}.

Let $\{X_i, X_{n,i};\, n\geq 1,\, i\geq 1\}$ be random variables taking values in $[0,1]$. Set $X_n=(X_{n1}, X_{n2},\cdots)\in [0,1]^{\infty}$. If $X_{ni}=0$ for $i>m$, we simply write $X_n=(X_{n1},\cdots, X_{nm})$. We say that $X_n$ \emph{converges weakly} to $X:=(X_1, X_2, \cdots)$ as $n\to\infty$ if, for any $r\geq 1$, $(X_{n1},\cdots, X_{nr})$ converges weakly to $X=(X_1, \cdots, X_r)$ as $n\to\infty$. This convergence actually is the same as the weak convergence of random variables in $([0,1]^{\infty}, d)$ where
\bea\lbl{Pepsi}
d(x,y)=\sum_{i=1}^{\infty}\frac{|x_i-y_i|}{2^i}
\eea
for $x=(x_1, x_2, \cdots)\in [0, 1]^{\infty}$ and $y=(y_1, y_2, \cdots)\in [0, 1]^{\infty}$. The topology generated by this metric is the same as  the product topology.

\begin{lemma}\label{lem:infinity}
 Let $m=m_n\to\infty$ as $n\to\infty.$ Let $\kappa=(k_1,\ldots, k_m) \in \ml{P}_n(m)$ be chosen with probability as in \eqref{eq:general} under the assumption of Theorem \ref{thm:general-infinity}. Let $(X_{m,1}, \cdots, X_{m,m})$ and $X=(X_1, X_2, \cdots)$ be random variables taking values in $\nabla_{m-1}$ and  $\nabla$, respectively. If
\bea\lbl{drink}
\sup_{\|\varphi\|_L\leq 1}\Big|E\varphi\Big(\frac{k_1}{n}, \cdots, \frac{k_m}{n}\Big) -
  E\varphi(X_{m,1}, \cdots, X_{m,m})\Big| \to 0
\eea
as $n\to\infty$, and $(X_{m,1}, \cdots, X_{m,m})$ converges weakly to $X$ as $n\to\infty$, then $\big(\frac{k_1}{n}, \cdots, \frac{k_m}{n}\big)$ converges weakly to $X$ as $n\to\infty$.
\end{lemma}
\begin{proof} Given integer $r\geq 1$, to prove the theorem, it is enough to show $\big(\frac{k_1}{n}, \cdots, \frac{k_r}{n}\big)$ converges weakly to $(X_1, \cdots, X_r)$ as $n\to\infty.$ Since $m=m_n\to\infty$ as $n\to\infty$, without loss of generality, we assume $r<m$ in the rest of discussion. For any random vector $Z$, let $\ml{L}(Z)$ denote  its  probability distribution. Review (\ref{cream}). By the triangle inequality,
\bea
& & \rho\Big(\mathcal{L}\big(\frac{k_1}{n}, \cdots, \frac{k_r}{n}\big),\, \mathcal{L}\big(X_1, \cdots, X_r\big)\Big)\nonumber\\
& \leq & \rho\Big(\mathcal{L}\big(\frac{k_1}{n}, \cdots, \frac{k_r}{n}\big),\, \mathcal{L}\big(X_{m,1}, \cdots, X_{m,r}\big)\Big) + \rho\Big(\mathcal{L}\big(X_{m,1}, \cdots, X_{m,r}\big),\, \mathcal{L}\big(X_1, \cdots, X_r\big)\Big) \nonumber\\
& & ~~~~~~~~~~~~~~~~~~~~~~~~~~~~~~~~~~~~~~~~~~~~~~~~~~~~~~~~~~~~~~~~~~~~~~~\lbl{tang}
\eea
For any function $\varphi(x_1, \cdots, x_r)$ defined on $[0, 1]^r$ with $\|\varphi\|_{L}\leq 1$, set $\tilde{\varphi}(x_1, \cdots, x_m)=\varphi(x_1, \cdots, x_r)$ for all $(x_1, \cdots, x_m)\in\mathbb{R}^m$. Then $\|\tilde{\varphi}\|_{L}\leq 1$. Condition (\ref{drink}) implies that the middle one among the three distances in (\ref{tang}) goes to zero. Further, the assumption that  $(X_{m,1}, \cdots, X_{m,m})$ converges weakly to $X$ implies the third distance in (\ref{tang}) also goes to zero. Hence the first distance goes to zero. The proof is completed.
\end{proof}

With Lemma \ref{lem:infinity} and the estimation in Theorem \ref{thm:general}, we obtain the proof of Theorem \ref{thm:general-infinity}.
\begin{proof}[Proof of Theorem \ref{thm:general-infinity}]
Assume $\kappa=(k_1,\ldots, k_m) \in \ml{P}_n(m)$ is chosen with probability as in \eqref{eq:general}. The proof is almost identical to the proof of Theorem \ref{thm:general}. We only mention the difference and modifications. Instead of choosing the test function $\psi$ to be bounded and continuous as in the beginning of Theorem \ref{thm:general}, we select $\psi=\varphi$ to be bounded and Lipschitz. Following the proof of Theorem \ref{thm:general}, the function $G$ defined in \eqref{def:G} in \emph{Step 2} is now bounded and Lipschitz on $\overline{\nabla}_{m-1}$. The major change happens in \emph{Step 3}, where we replace the estimation in \eqref{eq:bdG} by
\begin{align*}
 \Big|G\Big(\frac{k_1}{n},\ldots,\frac{k_{m-1}}{n}\Big) - G(y_1,\ldots,y_{m-1})\Big|
 &\le C \cdot \sqrt{\sum_{i=1}^{m-1}\Big(y_i - \frac{k_i}{n}\Big)^2}\\
 &\le C \cdot \frac{\sqrt{m}}{n},
\end{align*}
for some constant $C$ depending only on the Lipschitz constant of $G$, where $y_i \in [\frac{k_i-1}{n}, \frac{k_i}{n}]$ for $1\le i \le m-1$. Consequently, the term $\mathcal{S}_1$ defined in the end of \emph{Step 2} is now bounded as follows:
\bea\lbl{eq:bdS1}
|\mathcal{S}_1| &\le& \sum_{(k_1,\ldots,k_{m-1}) \in \{1,\ldots,n\}^{m-1}} \int_{\frac{k_1-1}{n}}^{\frac{k_1}{n}} \cdots \int_{\frac{k_{m-1}-1}{n}}^{\frac{k_{m-1}}{n}} \nonumber\\
& & \quad \quad \quad \quad \quad \quad  \quad  \Big|G\Big(\frac{k_1}{n},\ldots,\frac{k_{m-1}}{n}\Big) - G(y_1,\ldots,y_{m-1})\Big|  ~d y_1 \dots d y_{m-1} \nonumber\\
&\le& C \cdot \frac{\sqrt{m}}{n} \Big(\frac{1}{n}\Big)^{m-1} n^{m-1} = \frac{C\sqrt{m}}{n}.
\eea
\emph{Step 4} remains the same and we modify \emph{Step 5} using the changes mentioned above. The difference between  the numberators in \eqref{eq:part} and \eqref{eq:gen} now becomes
\bea\lbl{eq:diff1}
|\mathcal{I}_1 - \mathcal{I}_2| \le |\mathcal{S}_1| + |\mathcal{S}_2| \le C_1\cdot \left(\frac{\sqrt{m}}{n} + \frac{(1+\frac{m}{n})^m}{n} \right)
\eea
as $n\to \infty$ for some constant $C_1$ depending only on the Lipschitz constants of $\varphi$ and $f$ and the upper bounds of $\varphi$ and $f$ on the compact set $\overline{\nabla}_{m-1}$. Using the same argument in the end of the proof of Theorem \ref{thm:general} and the assumption that $\|f\|_{Lip}\le K$, we have for any $\varphi$ defined on $\nabla_{m-1}$ satisfying $\|\varphi\|_L \le 1$,
\beaa
& & \sup_{\|\varphi\|_L \le 1}|\E\left(\varphi(\frac{k_1}{n},\ldots,\frac{k_m}{n}) \right) - \E\left(\varphi(X_{m,1},\ldots,X_{m,m}) \right)|\\
&=& O\Big(\frac{\sqrt{m}}{n} + \frac{(1+\frac{m}{n})^m}{n}\Big) + |\mathcal{E}_{n,m}| \to 0.
\eeaa
as $n\to \infty$. Recall in \eqref{eq:error}, we have $|\mathcal{E}_{n,m}| \to 0$ as long as $m=o(n^{1/3})$. Therefore, by Lemma \ref{lem:infinity}, we conclude that $(\frac{k_1}{n},\cdots, \frac{k_m}{n})$ converges weakly to $X$ as $n\to \infty$.
\end{proof}

\bibliography{PD}
\bibliographystyle{apalike}


\end{document}